\documentclass[a4paper,11pt,oneside]{article}

\usepackage[utf8]{inputenc}
\usepackage[T1]{fontenc}
\usepackage[english]{babel}

\usepackage{answers}
\usepackage{setspace}
\usepackage{graphicx}
\usepackage{enumitem}
\usepackage{multicol}
\usepackage{mathrsfs}
\usepackage{amsmath,amsthm,amssymb}
\usepackage{bbm}
\usepackage{esdiff}
\usepackage{braket}
\usepackage{tikz}
\usepackage{dsfont}
\usepackage[hidelinks]{hyperref}
\usepackage{bm}
\usepackage{comment}
\usepackage{dirtytalk}
\usepackage{mathtools,tikz-cd}

\usepackage{caption}
\usepackage{subcaption}

\usepackage{color,soul}

\usepackage{cleveref}

\usepackage{nicefrac, xfrac}

\newtheorem{innercustomgeneric}{\customgenericname}
\providecommand{\customgenericname}{}
\newcommand{\newcustomtheorem}[2]{%
  \newenvironment{#1}[1]
  {%
   \renewcommand\customgenericname{#2}%
   \renewcommand\theinnercustomgeneric{##1}%
   \innercustomgeneric
  }
  {\endinnercustomgeneric}
}

\newcustomtheorem{customassumptions}{Assumptions}
\newcustomtheorem{customassumption}{Assumption}

\makeatletter
\newtheorem*{rep@theorem}{\rep@title}
\newcommand{\newreptheorem}[2]{%
\newenvironment{rep#1}[1]{%
 \def\rep@title{#2 \ref{##1}}%
 \begin{rep@theorem}}%
 {\end{rep@theorem}}}
\makeatother

\newreptheorem{theorem}{Theorem}
\newreptheorem{lemma}{Lemma}

\usepackage[toc,page]{appendix}

\newcommand{\nocontentsline}[3]{}
\newcommand{\tocless}[2]{\bgroup\let\addcontentsline=\nocontentsline#1{#2}\egroup}

\usepackage[a4paper,top=3cm,bottom=3cm,left=2.5cm,right=2.5cm]{geometry}
\numberwithin{equation}{section}

\newcommand{\N}{\mathbb{N}}

\newcommand{\R}{\mathbb{R}}

\newcommand{\E}{\mathbb{E}}

\newcommand{\J}{\mathfrak{J}}

\newcommand{\F}{\mathcal{F}}
\newcommand{\prob}{\mathbb{P}}

\newcommand{\tonde}[1]{\left({#1}\right)}
\newcommand{\quadre}[1]{\left[{#1}\right]}

\newcommand{\abs}[1]{\left\lvert{#1}\right\rvert}

\newcommand{\eps}{\varepsilon}

\newcommand{\innprod}[1]{\langle {#1} \rangle}

\theoremstyle{plain}
\newtheorem{thm}{Theorem}[section]
\newtheorem{lemma}[thm]{Lemma}
\newtheorem{prop}[thm]{Proposition}

\theoremstyle{definition}
\newtheorem{definition}{Definition}

\newtheorem{remark}{Remark}

\usepackage{titling}
\thanksmarkseries{alph}

\usepackage{nameref}

\title{Coarse correlated equilibria in linear quadratic mean field games and application to an emission abatement game}
\date{\today}

\author{Luciano Campi\thanks{Department of Mathematics ``Federigo Enriques'', University of Milan, Via Saldini 50, 20133, Milan, Italy. E-mail address: \href{mailto:luciano.campi@unimi.it}{luciano.campi@unimi.it}} \and Federico Cannerozzi\thanks{Department of Mathematics ``Federigo Enriques'', University of Milan, Via Saldini 50, 20133, Milan, Italy. E-mail address: \href{mailto:federico.cannerozzi@unimi.it}{federico.cannerozzi@unimi.it}} \and Fanny Cartellier\thanks{Centre of Research in Economics and Statistics (CREST), CNRS, École polytechnique, GENES, ENSAE Paris, Institut Polytechnique de Paris, 91120 Palaiseau, France. E-mail address: \href{mailto:fanny.cartellier@ensae.fr}{fanny.cartellier@ensae.fr}.
} } 

\begin{document}

\maketitle

\begin{abstract}

Coarse correlated equilibria (CCE) are a good alternative to Nash equilibria (NE), as they arise more naturally as outcomes of learning algorithms and they may exhibit higher payoffs than NE.  
CCEs include a device which allows players' strategies to be correlated without any cooperation, only through information sent by a mediator. 
We develop a methodology to concretely compute mean field CCEs in a linear-quadratic mean field game framework. We compare their performance to mean field control solutions and mean field NE (usually named MFG solutions). 
Our approach is implemented in the mean field version of an emission abatement game between greenhouse gas emitters. In particular, we exhibit a simple and tractable class of mean field CCEs which allows to outperform very significantly the mean field NE payoff and abatement levels, bridging the gap between the mean field NE and the social optimum obtained by mean field control. 

\smallskip


\end{abstract}

\section{Introduction}
Mean field games (MFGs) have been introduced in mid 2000s in \cite{lasry_lions} and independently in \cite{huang_malhame_caines}.
They arise as limit systems of large dynamic symmetric games with interactions of mean field type.
In the limit, the concept of Nash equilibrium translates into a fixed point problem in the space of flows of measures. This equilibrium concept is commonly defined as an MFG solution, for two main reasons.
On the one hand, approximate Nash equilibria with vanishing approximation error can be constructed starting from such an MFG solution (see, e.g., \cite{campi18absorption,car_del_probabilistic,lacker_leflem2022}).
On the other hand, Nash equilibria (NEs) for the $N$-player game can be shown to converge to such MFG solutions (see, e.g., \cite{lacker2020convergence,lacker_leflem2022}).
In this sense, MFG solutions can be considered as the infinitely many players analogue of Nash equilibria,
so that one can (and we will) refer to commonly called MFG solutions also as mean field Nash equilibria (mean field NE, for short).

Despite their popularity, Nash equilibria present some flaws. First, they raise numerical complexity issues, see for instance \cite{gilboa1989nash}. Second, it is well-known in game theory that agents are proved to actually behave according to a Nash equilibrium only under strong rationality assumptions. 
Finally, they can be highly inefficient compared to social optimum. 
As an alternative to Nash equilibria, correlated equilibria (CEs) and coarse correlated equilibria (CCEs) have been introduced in game theory literature. 
They can be understood as a generalization of the notion of Nash equilibrium by the introduction of a correlation device, which allows agents to adopt correlated strategies without any cooperation. While CEs were introduced by Aumann in 1974 \cite{aumann1974}, CCEs were introduced in \cite{hannan1957} and explicitly by \cite{moulin_vial1978} as a generalization of CEs. 
CCEs have been shown in game theory and computational literature to arise naturally from no-regret adaptive learning procedures (\cite{hart_mascolell_regret_based},\cite[Section 17.4]{roughgarden_2016}). 
Moreover, they are computationally ``easier'' as shown by \cite{gilboa1989nash}. 
Finally, they are shown to be able to outperform NE payoffs in standard game theory \cite{DokkaTrivikram2022Edia,moulin2014improvingNashbyCC} even in situations where correlated equilibria cannot, for instance in potential games \cite{neyman1997correlated}. For these reasons in this paper we focus on CCEs.

CCEs can be interpreted as follows in an $N$-player setting. A moderator, or correlation device, picks a strategy profile for the $N$ players randomly according to some publicly know distribution; then, she recommends it privately to the players.
Before the lottery is run, each player has to decide whether to commit to the moderator recommendation (whatever it will be), assuming that all other players commit, only knowing the lottery distribution.
If a player commits, then she is communicated in private her (and only her) selected strategy, and must follow it.
Instead, if a player deviates, she will do so without any information on the outcome of the lottery, assuming that all other players follow the private suggestion they receive.
A lottery is a CCE if every player prefers to commit rather than unilaterally deviate, assuming that all others do commit. 
CCEs are a generalization of Aumann's notion of correlated equilibria (see \cite{aumann1974,aumann1987}), since in the latter each player is asked to commit to moderator's lottery after having seen her suggested strategy.

Lately, correlated and coarse correlated equilibria have made their appearance in MFG literature. 
Bonesini et al. \cite{bonesini_thesis,bonesiniCE,campifischer2021} establish existence and convergence results for correlated equilibria in mean field games with discrete time and finite state and action spaces. 
A second group of papers by Muller et al. \cite{muller2022learningCE,muller2021learning} considers both CEs and CCEs in a similar setting. In addition, they provide an extensive discussion of learning algorithms for both types of equilibria in MFG. Lastly, in \cite{campi_cannerozzi_fischer2023coarse}, CCEs have been introduced in both continuous time stochastic differential games and mean field games. The notion of coarse correlated solution to the MFG is justified by proving an approximation result. An existence result is also proved, by means of a minimax theorem. Although its generality, this result is not constructive, and the question of how to construct coarse correlated solutions to MFG is left open. 

\medskip
This paper's goal is to develop a methodology for computing mean field CCEs, and to effectively compare them to mean field NEs and mean field control (MFC) solutions (see \cite{carmona2013control,carmona2017priceofanarchy} for an insightful discussion on the differences between such two notions and a quantitative comparison).
For this reason, we do not consider the $N$-player game, but we limit our analysis to the mean field game.
Since we search for explicit solutions, we restrict our analysis to linear-quadratic stochastic MFGs, working in a setting closely related to \cite{graber2016LinearQuadratic}.
Applying our methodology to a toy model, we show that mean field CCEs indeed allow to significantly outperform the mean field NE in terms of payoffs under identified conditions. 

We propose a notion of mean field CCE which is strongly inspired by the notion of coarse correlated solution to the MFG of \cite{campi_cannerozzi_fischer2023coarse}. 
As for a mean field NE, our notion of mean field CCEs is any suitable pair made of a strategy and a flow of moments, with the following important differences. The flow of moments can be stochastic, and the strategy can be correlated to the flow of moments even without the presence of a common noise, as it is the case in this paper. The way they are correlated is chosen by the moderator at the beginning of the game as part of the equilibrium. We call such pairs \emph{correlated flows}. In few words, any of such correlated flows is a mean field CCE if the representative player has no incentive to deviate before knowing the flow realisation, and if the flow is consistent, i.e., at any time $t$ the flow of moments equals the conditional expectation of the representative player's state given the whole flow of moments up to terminal time. 

\medskip
Our main contributions can be summarised as follows:
\begin{itemize}[label= --]
\item After focusing on a suitable class of suggested strategies and flows of moments verifying the consistency condition, we reduce the search of a mean field CCE to an inequality involving only the law of stochastic flow of moments at the equilibrium.
\item We compare the payoffs of mean field CCEs with those of mean field NEs and MFC solutions. We show that the MFC optimal payoff is the unattainable upper bound for all mean field CCEs and provide a condition on the law of the stochastic flow of moments so that mean field CCEs in a specific class yield a higher payoff than mean field NE.
\item Finally, we apply our results to an emission abatement game between countries, inspired by environmental economics literature on international environmental agreements  \cite{barrett1994self,DokkaTrivikram2022Edia}. 
We show that it is possible to build simple mean field CCEs that both yield much higher payoffs than the mean field NE and guarantee higher average abatement levels. 
\end{itemize}
The application also shows an additional interest of CCEs, which is to help a regulator not only to lead the population to a more optimal payoff than the free-riding NE, but also or otherwise to lead it to match other and potentially payoff-conflicting targets, such as the abatement level of players in this application. To the best of our knowledge, no attempt has been made so far to identify CCEs analytically in a mean field game, nor to explore and illustrate their potential in outperforming the payoffs of mean field Nash equilibria.

\medskip
The rest of the paper is organised as follows: in Section \ref{sec:setting} we state the assumptions, which will be in force throughout the whole paper, and give the definition of mean field CCE.
In Section \ref{sec:cce}, we develop the methodology for computing mean field CCEs, while in Section \ref{sec:comparison} the comparison between mean field CCEs, MFC solutions and mean field NEs is carried out.
In Section \ref{sec:AG}, we apply the results of the previous sections to the abatement game, and we analyse and explore the resulting characterization of the set of mean field CCEs which outperform the payoff of the unique mean field NE.
Finally, we collect in the \nameref{sec:appendix} the most standard proofs, which we choose to include for the sake of completeness.

\section{Setting}\label{sec:setting}
Let $T > 0$ be a fixed time horizon.
Let $d,k \in \N$.
For $n \in \N$, denote by $\mathcal{S}^n$ the set of $n \times n$ symmetric matrices and by $I_n$ the identity matrix in $\mathcal{S}^n$. We are going to work under the following set of assumptions.

\begin{customassumptions}{\textbf{A}}\label{LQ:assumptions}
Consider the following vector valued or matrix valued functions: 
\begin{enumerate}[label=(\arabic*),wide]
    \item $A,\sigma \in L^\infty([0,T];\R^{d \times d})$; 
    \item $B \in L^\infty([0,T];\R^{d \times k})$;
    \item $Q,\bar Q, \tilde Q \in L^\infty([0,T];\mathcal{S}^{d})$, $R \in \mathcal{C}([0,T];\mathcal{S}^k)$, $H, \bar H, \tilde H \in \mathcal{S}^d$; 
    \item $H, \bar H, \tilde H \geq 0$, $Q_t \geq d_1 I_{d}$ for every $t \in [0,T]$, $d_1 \geq 0$, $R_t \geq d_2 I_{k}$ for every $ t \in [0,T]$, $d_2 > 0$;
    \item $S \in L^\infty([0,T];\R^{k \times d})$, $\sup_{t \in [0,T]} \vert S_t \vert^2  < d_1d_2$ if $d_1 > 0$, $S_t = 0$ for every $t \in [0,T]$ otherwise;
    \item $L,q \in L^\infty([0,T];\R^{d})$, $r \in L^\infty([0,T];\R^{k})$.
\end{enumerate}
\end{customassumptions}

Let $(\Omega,\F,\mathbb{F},\prob)$ be a complete filtered probability space satisfying usual assumptions, let $W$ be a $d$ dimensional $\mathbb{F}$-Brownian motion and let $\xi$ be an $\R^d$-valued $\F_0$-measurable random variable with law $\nu$. 
Denote by $\nu_1$ and $\nu_2$ the first and second moments of $\nu$ respectively.
Suppose that $\xi$ and $W$ are independent.
Throughout the paper, we assume the following assumption:
\begin{customassumption}{\textbf{U}}\label{LQ:assumption_F_0}
The $\sigma$-algebra $\mathcal F_0$ is large enough to support a $\mathcal F_0$-measurable uniform random variable independent of $\xi$ and $W$.
\end{customassumption}
In the following, we denote by $\mathbb{F}^1=(\F^1_t)_{t \in [0,T]}$ the filtration generated by $\xi$ and $W$, which we assume without loss of generality to satisfy the usual conditions.

Given an arbitrary filtration $\mathbb G$, we will use the standard notation $\mathbb H^2(\mathbb G)$ for the set of all $\mathbb G$-progressively measurable $\mathbb R^k$-valued processes $\alpha=(\alpha_t)_{t\in [0,T]}$ such $\mathbb E[\int_0 ^T | \alpha_t |^2 dt]<\infty$.

\medskip
We introduce the notion of correlated flow and mean field coarse correlated equilibrium.
\begin{definition}[Correlated flow]\label{def_corr_flow}
A correlated flow is a pair $(\lambda,\mu)$ satisfying the following properties:
\begin{enumerate}[label=\roman*)]
    \item $\lambda=(\lambda_t)_{t \in [0,T]}$ is a process in $\mathbb H^2(\mathbb F)$.
    \item $\mu=(\mu_t)_{t \in [0,T]}$ is an $\F_0$-measurable $\mathcal{C}([0,T];\R^d)$-random variable.
    \item $\mu$ is independent of both $\xi$ and $W$.
\end{enumerate}
We refer to $\lambda$ as the \textit{recommended strategy} and to $\mu$ as the \textit{random flow of moments}. 
\end{definition}

We can interpret a correlated flow $(\lambda,\mu)$ as follows: moderator's lottery is run before the game starts and independently of the idiosyncratic shocks that determine the random evolution of representative player's state. This is made possible by Assumption \ref{LQ:assumption_F_0}, which allows for some independent extra randomness. We stress that, while the recommended strategy $\lambda$ is correlated both to $\xi$ and $W$ and to $\mu$, $\mu$ is independent of the initial datum and the noise. We will sometimes use the equivalent expressions ``correlated strategy'' or ``suggested strategy'' to refer to $\lambda$.

\medskip
Let us consider a correlated flow $(\lambda,\mu)$. We now assign dynamics and payoff functional.
We consider a state variable with linear dynamics given by
\begin{equation}\label{LQ_general:dynamics}
dX_t = (A_t X_t + B_t \lambda_t)dt + \sigma_tdW_t, \quad X_0=\xi,
\end{equation}
and a linear-quadratic payoff functional
\begin{equation}\label{LQ_general:payoff}
\begin{aligned}
\J & (\lambda,\mu)= \E \bigg[ \int_0^T \Big( \big( \innprod{L_t,\mu_t} -\frac{1}{2}\innprod{\bar Q_t \mu_t, \mu_t }\big) - \big(  \frac{1}{2}\innprod{ Q_t X_t, X_t }  + \innprod{\Tilde{Q}_t X_t, \mu_t } +\innprod{q_t,X_t} + \frac{1}{2}\innprod{R_t\lambda_t,\lambda_t}  \\
& + \innprod{S_t X_t, \lambda_t } + \innprod{r_t,\lambda_t}  \big) \Big) dt - \frac{1}{2}\innprod{ \bar H \mu_T, \mu_T } - \big( \frac{1}{2}\innprod{ H X_T, X_T } + \innprod{\Tilde{H} X_T, \mu_T } \big) \bigg].
\end{aligned}
\end{equation}
When needed, we will stress the dependence of the process $X$ on the control $\lambda$ by using the notation $X^{\lambda}$.

Now, in order to move to the definition of mean field CCE, two cases must be distinguished.
If the representative player decides to trust the mediator and therefore accepts to follow her recommendation $\lambda$ before knowing it, the dynamics is given by equation \eqref{LQ_general:dynamics}, and the player gets the reward $\J(\lambda,\mu)$.
If instead she decides to deviate, she uses a strategy $\beta\in \mathbb H^2(\mathbb F ^1)$, her state dynamics is given by equation \eqref{LQ_general:dynamics} with $\beta$ instead of $\lambda$, and her reward is $\J(\beta,\mu)$.
Observe that when she deviates, her strategy $\beta$ is measurable only with respect to the initial datum and the idiosyncratic noise, since she has no information on the outcome of the moderator's lottery. The deviating player can only use her knowledge of the law of the correlated flow $(\lambda,\mu)$, which is assumed to be publicly known. As a consequence, when deviating, the state process $X$ of the representative player is independent of the random flow of moments $\mu$, which, however, still appears in her payoff.

\begin{definition}[Mean field coarse correlated equilibrium]\label{def_cce}
A correlated flow $(\lambda,\mu)$ is a mean field CCE  if the following holds:
\begin{enumerate}[label=(\roman*)]
    \item Optimality: for every deviation $\beta \in \mathbb H^2(\mathbb F ^1)$, it holds
    \begin{equation}\label{LQ:def_cce:opt}
        \J(\lambda,\mu) \geq \J(\beta,\mu).
    \end{equation}
    \item Consistency: let $X=(X_t)_{t \in [0,T]}$ be the solution to equation \eqref{LQ_general:dynamics} with the control process $\lambda$. For every time $t \in [0,T]$, $\mu_t$ is a version of the conditional expectation of $X_t$ given $\mu$, that is,
    \begin{equation}\label{LQ:def_cce:cons}
        \mu_t=\E[X_t \vert \mu ] \quad \prob\text{-a.s.} \;\; \forall t \in [0,T].
    \end{equation}
\end{enumerate}
\end{definition}
The definition of mean field CCE has two fundamental differences with the usual definition of mean field NE.
First of all, as already mentioned, the optimality condition features an asymmetry between the suggested strategy, which belongs to $\mathbb H^2(\mathbb F)$, and deviating player's strategies, which belong to the smaller class $\mathbb H^2(\mathbb F^1)$, since the former depends also on the information used by the moderator to run her lottery while the latter does not.
As for the consistency condition, we notice that, coherently with $\mu$ being stochastic, it is formulated in terms of conditional expectations, although no common noise is present.
It should be interpreted in the following way: if all players commit to the mediator's lottery outcomes before knowing them, then the flow of measures should arise from aggregation of the individual behaviors.
In the mean field limit, the influence of the idiosyncratic noise on the flow of moments vanishes, while the influence of moderator's lottery does not.
Therefore, $\mu$ stays stochastic and its stochasticity should derive from moderator's lottery only.
We refer to \cite{campi_cannerozzi_fischer2023coarse} for more considerations and a deeper analysis of the connection with the $N$-player game.

\begin{remark}\label{LQ_general:rmk:dynamics_no_moments}
The reader might have noticed that $\mu$ does not appear in the state dynamics \eqref{LQ_general:dynamics}. While computing mean field NEs and MFC solutions in the linear-quadratic case with the flow of moments in the dynamics is standard, computing mean field CCEs can be more delicate when $\mu$ appears in the state dynamics. We refer to Section \ref{sec:deviating_player} and to Remark \ref{LQ_general:rmk:filtering} therein for more explanations.
\end{remark}

\begin{remark}
In \cite{campi_cannerozzi_fischer2023coarse}, moderator's lottery was modeled in the following way: an auxiliary probability space was chosen by the moderator to support the extra randomness for her lottery.
As a consequence, the recommended strategy, dynamics and payoff were naturally defined on a suitable product space supporting $\xi, W$ and such extra randomness.
Here, thanks to Assumption \ref{LQ:assumption_F_0}, the given filtration $\mathbb F$ is already big enough to allow for any extra randomization the moderator might want to use. In both formulations moderator's lottery is run independently of $\xi$ and $W$ and deviations are measurable with respect to $\xi$ and $W$ only.
Moreover, while \cite{campi_cannerozzi_fischer2023coarse} considers a stochastic flow of measures, and the consistency condition is given in terms of conditional probabilities, here it is enough to consider a flow of moments and conditional expectations, due to the linear-quadratic structure of the MFG.
\end{remark}

\section{Computing mean field coarse correlated equilibria}\label{sec:cce}

The set of coarse correlated equilibria is typically very wide and it is difficult to characterize in a continuous time setting. We therefore focus on a tractable class of correlated flows for which we are able to characterize a sufficient condition for being a mean field CCE. To do so, we adopt the following procedure:
\begin{itemize}
    \item We fix a correlated flow $(\lambda,\mu)$.
    We suppose that the representative player does not commit to the moderator's lottery and we compute her best deviating strategy $\hat{\beta}$, i.e.
    \[
    \hat{\beta} = \text{arg}\max_{\beta \in \mathbb H^2(\mathbb F^1)}\J(\beta,\mu).
    \]
    This is the content of Proposition \ref{LQ_general:prop_best_dev}. Observe that $\hat{\beta}$ will depend upon the law of $(\lambda,\mu)$ itself, but not on its actual realization.
    \item We define a parameterised class of correlated flows $(\lambda,\mu)$ of similar shape as the best deviating strategy $\hat{\beta}$ so that the consistency condition \eqref{LQ:def_cce:cons} is fulfilled.
    The correlation is due to a suitable random parameter $\delta$. 
    This is accomplished in subsection \ref{sec:corr_flow_defined}. 
    \item Finally, for $(\lambda,\mu)$ in such a class, with corresponding parameter $\delta$, we express the optimality condition
    \[
    \J(\lambda,\mu) \geq \J(\hat{\beta},\mu)
    \]
    as an inequality involving the law of $\mu$ and $\delta$ only.
    Such an inequality is established in Theorem \ref{LQ:prop:optimality_condition}.
\end{itemize}
As a result, we reduce the search for a mean field CCE to finding a law for $\mu$ and $\delta$ that verifies an optimality inequality.
The choice of focusing on a class of correlated flows with shape similar to the best deviation allows for explicit analytical comparison between the two payoffs in the optimality condition \eqref{LQ:def_cce:opt}.

\begin{remark}
Interestingly, the outlined procedure does not involve the usual two steps procedure used to compute mean field NEs: first, optimize with a fixed flow of moments and, second, perform a fixed point argument to determine the flow. Indeed, we first impose the consistency condition and then we verify the optimality condition, more in line with an MFC fashion.
This sheds light on one important feature of mean field CCEs: they can be regarded as a middle ground between mean field NE and MFC solutions.
The comparison will be carried out in Section \ref{sec:comparison}, and in Section \ref{sec:AG} through the study of a simple yet important example.
\end{remark}

\subsection{Deviating player's optimization problem}\label{sec:deviating_player}
Suppose that the representative player does not commit to the lottery. Therefore, as anticipated in Section \ref{sec:setting}, she chooses a strategy on her own before the moderator sends his recommendation, hence in particular without any information on the realisation of the correlated flow. The only information she has about $(\lambda,\mu)$ is the joint law of the pair itself, which is assumed to be publicly known.
Due to the linear-quadratic structure of the MFG and the fact that any admissible deviation $\beta$ is independent of $\mu$, it turns out that knowing the expectation of $\mu_t$ for all $t\in [0,T]$ is enough.

Since the term $\int_0^T(\innprod{L_t,\mu_t} -\frac{1}{2}\innprod{\bar Q_t \mu_t, \mu_t })dt - \innprod{\bar H \mu_T, \mu_T }$ in \eqref{LQ_general:payoff} can be viewed as an uncontrolled constant for the deviating player's optimization problem, we can focus on the equivalent optimization problem
\begin{equation*}
\begin{aligned}
    \min_{\beta \in \mathbb H^2 (\mathbb F^1)} \J'(\beta,\mu),
\end{aligned}
\end{equation*}
where
\begin{equation}\label{LQ_general:cost_functional}
\begin{aligned}
\J' & (\beta,\mu) = \E \bigg[ \int_0^T \Big( \frac{1}{2}\innprod{ Q_t X_t, X_t }  + \innprod{\Tilde{Q}_t X_t, \mu_t } + \innprod{q_t,X_t} + \frac{1}{2}\innprod{R_t\beta_t,\beta_t} + \innprod{S_t X_t, \beta_t } + \innprod{r_t,\beta_t}   \Big) dt \\
& + \frac{1}{2}\innprod{ H X_T, X_T } + \innprod{\Tilde{H} X_T, \mu_T } \bigg]
\end{aligned}
\end{equation}
under the constraint
\begin{equation}\label{LQ_general:dev_player_dynamics}
dX_t = (A_t X_t + B_t \beta_t)dt + \sigma_t dW_t, \quad X_0=\xi.
\end{equation}
Since $\beta \in \mathbb H^2(\mathbb F^1)$, it follows that $X$ is $\mathbb{F}^1$-adapted, and therefore is independent of the flow of moments $\mu$, which implies that deviating player's payoff can be written as:
\begin{equation}\label{LQ:dev_player:only_cost_terms}
\begin{aligned}
\J' & (\beta,\mu) = \int_0^T \E \bigg[ \E \Big[ \frac{1}{2}\innprod{ Q_t X_t, X_t }  + \innprod{\Tilde{Q}_t X_t, \mu_t } +\innprod{q_t,X_t} +\frac{1}{2} \innprod{R_t\beta_t,\beta_t} + \innprod{S_t X_t, \beta_t } \\
& + \innprod{r_t,\beta_t} \; \bigg \vert \; \F^1_t \Big] \bigg] dt +  \E \bigg[ \E \quadre{\frac{1}{2}\innprod{H X_T, X_T } + \innprod{\Tilde{H} X_T, \mu_T } \; \Big \vert \; \F^1_T }\bigg] \\
= & \E \bigg[ \int_0^T \Big( \frac{1}{2}\innprod{ Q_t X_t, X_t }  + \innprod{\Tilde{Q}_t \E[\mu_t] + q_t, X_t } +\frac{1}{2} \innprod{R_t\beta_t,\beta_t} + \innprod{S_t X_t, \beta_t } + \innprod{r_t,\beta_t} \Big) dt \\
& + \frac{1}{2}\innprod{HX_T,X_T} +\innprod{\Tilde{H} \E[\mu_T], X_T } \bigg].
\end{aligned}
\end{equation}
This is now a standard linear quadratic control problem, which can be solved by the stochastic maximum principle.
\begin{prop}[Optimal strategy for the deviating player]\label{LQ_general:prop_best_dev}
Let $\phi$, $\psi$ and $\theta$ be the solutions of the following ODEs:
\begin{equation}\label{LQ:riccati_eq}
\left\{ \begin{aligned}
    & \dot\phi_t + \phi_t A_t + A_t^\top \phi_t  + Q_t -( \phi_t B_t + S_t^\top ) R_t^{-1}(B_t^\top\phi_t  + S_t) = 0, && \phi_T = H, \\
    & \dot\psi_t + A_t^\top \psi_t + \Tilde{Q}_t - (\phi_t B_t +S^\top ) R_t^{-1}B_t^\top\psi_t = 0, && \psi_T = \Tilde{H}, \\
    & \dot\theta_t + \psi_t\frac{d\E[\mu_t]}{dt} + A_t^\top\theta_t + q_t - (\phi_t B_t + S^\top )R_t^{-1}(B_t^\top\theta_t + r_t) = 0, && \theta_T = 0. \\
\end{aligned} \right.
\end{equation}
There exists a unique optimal strategy for the deviating player, which is given by
\begin{equation}\label{LQ_general:best_deviation}
    \hat{\beta}_t = - R_t^{-1}((B_t^\top \phi_t  + S_t) X_t + B_t^\top \psi_t \E[\mu_t] + B_t^\top \theta_t + r_t ).
\end{equation}
\end{prop}
We postpone the proof to the \nameref{sec:appendix}.
We observe only that the optimal control is actually feedback in the state $X_t$ and in the expectation $\E[\mu_t]$.
Moreover, while the functions $\phi$ and $\psi$ do not depend upon $\mu$ or its expectation, the flow of expectations $\E[\mu_t]$ appears in the equation for $\theta$, through its time derivative $\frac{d\E[\mu_t]}{dt}$.

\begin{remark}\label{LQ_general:rmk:filtering}
This first step towards calculating mean field CCEs requires a filtering procedure, since the deviating player does not observe the actual realisation of $\mu$. If the dynamics of the deviating player were dependent on $\mu$, this step would require a much more involved analysis. Indeed, the state process $X$ and $\mu$ would not be independent, even if $\beta \in \mathbb H^2(\mathbb F^1)$, which would lead to considering the projections on $\mathbb F^1$ of the processes $X^i$, $X^i X^j$ and $X^j \mu^i$, $1 \leq i,j \leq d$. This is why we have opted for a flow-free state dynamics, and postponed the analysis of the more general case to future research.
\end{remark}

\subsection{Correlated flow}\label{sec:corr_flow_defined}
We now consider a class of correlated flows $(\lambda,\mu)$ with a similar structure as the deviating player's best strategy $\hat \beta$ in \eqref{LQ_general:best_deviation}.
Our goal is to easily compare the payoff functionals $\J'(\lambda,\mu)$ and $\J'(\hat{\beta},\mu)$.
Hence we use the same functions $\phi$ and $\psi$, whereas we replace $\E[\mu_t]$ with $\mu_t$ itself and the term $R_t^{-1} (B_t^\top\theta_t + r_t)$ with a free parameter $\delta=(\delta_t)_{t \in [0,T]}$. Given any such $\delta$, we define $\mu$ so that the consistency condition \eqref{LQ:def_cce:cons} is satisfied, so that we will be left with taking care of the optimality condition only. 

More precisely, let $\mathcal G$ be the set of all correlated flows $(\lambda,\mu)$ defined as
\begin{equation}\label{LQ:corr_flow}
\begin{aligned}
    & \lambda_t = - R_t^{-1}((B_t^\top \phi_t  + S_t) X_t + B_t^\top \psi_t \mu_t  + \delta_t ), \\
    & \dot{\mu}_t = (A_t - B_tR^{-1}_t(B_t^\top \phi_t  + S_t + B_t^\top \psi_t ))\mu_t - B_t R_t^{-1} \delta_t , \quad \mu_0 = \nu_1,
\end{aligned}
\end{equation}
where $\delta =(\delta_t)_{t \in [0,T]}$ is any process in $\mathbb H^2(\mathcal F_0)$ independent of $\xi$ and $W$, and $\phi$ and $\psi$ are as in \eqref{LQ:riccati_eq}.
The parameter $\delta$ represents the extra source of randomness in the correlated flow with respect to $\xi$ and $W$.

\begin{lemma}\label{LQ:lemma_consistency}
Any correlated flow $(\lambda,\mu) \in \mathcal G$ satisfies the consistency condition \eqref{LQ:def_cce:cons}.
\end{lemma}

\begin{proof}
Let $(\lambda, \mu) \in \mathcal G$ corresponding to some $\delta \in \mathbb H^2(\F_0)$ independent of $\xi$ and $W$.
To ease the notation, set
\begin{equation}\label{LQ:fb_coefficients}
\begin{aligned}
    & \Phi_t = R_t^{-1}(B_t^\top \phi_t  + S_t), && \Psi_t = R_t^{-1}B_t^\top \psi_t , &&& \Theta_t = R_t^{-1}(B_t^\top \theta_t + r_t).
\end{aligned}
\end{equation}
Notice that $\mu$ satisfies the measurability requests of Definition \ref{def_corr_flow}.
The dynamics of the representative player state is given by
\begin{equation}\label{LQ:lemma_consistency:dynamics}
\begin{aligned}
    & dX_t = \tonde{(A_t -B_t\Phi_t)X_t -B_t(\Psi_t \mu_t + R_t^{-1}\delta_t) }dt + \sigma_t dW_t, \\
    & X_0= \xi,
\end{aligned}
\end{equation}
which implies that the process $(\mu_t - X_t)_{t \in [0,T]}$ satisfies the stochastic differential equation
\begin{equation}\label{LQ:lemma_consistency:difference}
    d(\mu_t - X_t) = (A_t -B_t\Phi_t)(\mu_t - X_t)dt - \sigma_t dW_t,
\end{equation}
Since $\delta \in \mathbb H^2(\F_0)$, equation \eqref{LQ:lemma_consistency:dynamics} admits a unique continuous adapted solution $X$ satisfying $\E[\sup_{t \in [0,T]}\vert X_t \vert ^2]< \infty$.
Since $\xi$, $W$ and $\delta$ are independent by assumption, by taking the conditional expectation with respect to $\mu$ in \eqref{LQ:lemma_consistency:difference}, we get 
\begin{equation*}
    d\E[\mu_t - X_t \vert \mu ] = (A_t -B_t\Phi_t)\E[\mu_t - X_t \vert \mu ]dt, \quad \E[\mu_0 - X_0 \vert \mu ]= \mu_0 - \E[\xi]=0, \quad \prob\text{-a.s.},
\end{equation*}
which implies
$\E[\mu_t - X_t \vert \mu ] = 0$ $\prob$-a.s. for every $t$, i.e. \eqref{LQ:def_cce:cons}.
\end{proof}

\begin{remark}
Although the structure of the class $\mathcal G$ is simple and quite specific, we will see later in the application section (Section \ref{sec:AG}) that it is rich enough to contain a large set of mean field CCEs with some desirable properties, such as significantly outperforming the mean field NE.
\end{remark}

\subsection{Optimality condition}\label{sec:optimality_condition}
Let $(\lambda,\mu) \in \mathcal G$. Since consistency has already been verified in Lemma \ref{LQ:lemma_consistency}, the goal is now to restate the optimality condition \eqref{LQ:def_cce:opt} in terms of quantities dependent upon the law of $\mu$ and $\delta$ only.

\begin{thm}\label{LQ:prop:optimality_condition}
Let $(\lambda,\mu) \in \mathcal G$ corresponding to some $\delta \in \mathbb H^2(\F_0)$ independent of $\xi$ and $W$.
Let $\Phi$, $\Psi$ and $\Theta$ be given by in \eqref{LQ:fb_coefficients}.
Set
\begin{equation}\label{LQ:auxiliary_coefficients}
\begin{aligned}
    & M_t = Q_t + \Phi_t^\top R_t \Phi_t - 2\Phi_t^\top S_t, && N_t = \Tilde{Q}_t + \Psi_t^\top R_t \Phi_t -\Psi_t^\top S_t , &&& G_t = \Psi_t^\top R_t \Psi_t .
\end{aligned}
\end{equation}
Let $f(\mu)=(f_t(\mu))_{t \in [0,T]}$ be given by
\begin{equation}\label{LQ:auxiliary_f}
\left\{ \begin{aligned}
    & \dot{f}_t(\mu) =  (A_t - B_t \Phi_t )f_t(\mu) + B_t\tonde{ \Psi_t(\mu_t - \E[\mu_t]) + R_t^{-1}\delta_t - \Theta_t }, \quad 0 \leq t \leq T, \\
    & f_0(\mu) = 0.
\end{aligned} \right.
\end{equation}
Then, $(\lambda,\mu)$ is a mean field CCE if and only if the following condition is satisfied:
\begin{equation}\label{LQ:cce:final_optimality}
\begin{aligned}
    \int_0^T & \Big( \E[\innprod{N_t (\mu_t-\E[\mu_t]), \mu_t-\E[\mu_t]}]  + \frac{1}{2}(\E[\innprod{G_t \mu_t,\mu_t}] - \innprod{G_t\E[\mu_t],\E[\mu_t]}) \\
    & + \frac{1}{2} \E[\innprod{R_t^{-1}\delta_t ,\delta_t  }] -\frac{1}{2} \innprod{R_t\Theta_t,\Theta_t }  \Big) dt + \E[\innprod{\Tilde{H} (\mu_T-\E[\mu_T]), \mu_T-\E[\mu_T]}]  \\
    \leq &  \int_0^T \Big( \frac{1}{2}\left(\E[\innprod{M_t (\mu_t + f_t(\mu)),\mu_t + f_t(\mu)}] -\E[\innprod{M_t \mu_t,\mu_t}]\right) + \E[\innprod{N_t f_t(\mu), \E[\mu_t] }] \\
    & + \innprod{q_t - \Phi_t^\top r_t, \E[f_t(\mu)] }  + \innprod{B_t^\top(\phi_t + \psi_t)   \E[\mu_t] , \Theta_t}  -\E[\innprod{B_t^\top(\phi_t + \psi_t) \mu_t , R_t^{-1}\delta_t}] \\
    & + \E[\innprod{ B_t^\top\phi_t f_t(\mu) , \Theta_t }  ]   - \E[\innprod{r_t, \Theta_t - R_t^{-1}\delta_t  }] \Big) dt \\
    & + \frac{1}{2}(\E[\innprod{H (\mu_T + f_T(\mu)),\mu_T + f_T(\mu)}] -\E[\innprod{H \mu_T,\mu_T}])  + \innprod{\Tilde{H} \E[f_T(\mu)],\E[\mu_T]}.
\end{aligned}
\end{equation}
\end{thm}
\begin{proof}
Since $(\lambda,\mu) \in \mathcal G$ satisfies the consistency condition \eqref{LQ:def_cce:cons} by Lemma \ref{LQ:lemma_consistency}, we focus on the optimality condition \eqref{LQ:def_cce:opt}.
This is equivalent to verifying that
\begin{equation}\label{LQ:opt2}
    \J'(\lambda,\mu) \leq \J'(\hat{\beta},\mu),
\end{equation}
with $\hat{\beta}$ given by \eqref{LQ_general:best_deviation} and $\J'(\lambda,\mu)$ and $\J'(\hat{\beta},\mu)$ are defined by \eqref{LQ:dev_player:only_cost_terms}.
Denote by $\hat{X}=(\hat{X}_t)_{t \in [0,T]}$ the state of the deviating player when she uses the strategy $\hat{\beta}$ defined in \eqref{LQ_general:best_deviation}, i.e.
\begin{equation*}
d\hat{X}_t = ((A_t - B_t\Phi_t)\hat{X}_t -B_t( \Psi_t\E[\mu_t] +\Theta_t))dt +\sigma_t dW_t,\quad \hat X_0 = \xi,
\end{equation*}
and by $X=(X_t)_{t\in[0,T]}$ the state of the representative player corresponding to the correlated flow \eqref{LQ:corr_flow}, i.e.
\begin{equation*}
    dX_t = \tonde{(A_t -B_t\Phi_t)X_t -B_t(\Psi_t \mu_t + R_t^{-1} \delta_t }dt + \sigma_t dW_t, \quad X_0= \xi.
\end{equation*}
We rewrite the cost functionals $\J'(\lambda,\mu)$ and $\J'(\hat{\beta},\mu)$ by taking advantage of the explicit form of $\lambda$ and $\hat{\beta}$ and functions \eqref{LQ:auxiliary_coefficients}:
\begin{align*}
    \J'& (\lambda,\mu) = \E\bigg[ \int_0^T \Big(  \frac{1}{2}\innprod{M_t X_t,X_t} + \innprod{N_t X_t, \mu_t } +\frac{1}{2}\innprod{G_t \mu_t,\mu_t} +\innprod{q_t -\Phi_t^\top r_t, X_t }  \\
    & +\innprod{(R_t \Phi_t -S_t) X_t ,R_t^{-1}\delta_t  } +\innprod{ R_t\Psi_t  \mu_t ,R_t^{-1}\delta_t } -\innprod{\Psi_t^\top r_t, \mu_t } \\
    & +\frac{1}{2}\innprod{R_t^{-1}\delta_t ,\delta_t  } -\innprod{r_t, R_t^{-1}\delta_t } \Big) dt + \frac{1}{2}\innprod{H X_T, X_T} + \innprod{\Tilde{H}X_T, \mu_T} \bigg],
\end{align*}
and
\begin{align*}
    \J' & (\hat{\beta},\mu)= \E\bigg[ \int_0^T \Big(  \frac{1}{2}\innprod{M_t \hat{X}_t,\hat{X}_t} + \innprod{N_t \hat{X}_t, \E[\mu_t] } +\frac{1}{2}\innprod{G_t \E[\mu_t],\E[\mu_t]} +\innprod{q_t -\Phi_t^\top r_t, \hat{X}_t } \\
    & +\innprod{(R_t \Phi_t - S_t) \hat X_t , \Theta_t } +\innprod{ R_t \Psi_t \E[\mu_t] , \Theta_t }  -\innprod{\Psi_t^\top r_t, \E[\mu_t] } \\
    & +\frac{1}{2}\innprod{R_t \Theta_t,\Theta_t } -\innprod{r_t, \Theta_t} \Big) dt + \frac{1}{2}\innprod{H \hat{X}_T, \hat{X}_T} + \innprod{\Tilde{H}\hat{X}_T, \E[\mu_T]} \bigg]
\end{align*}
By It\^o's formula, we get
\begin{equation*}
\begin{aligned}
    d(\hat{X}_t - X_t) = (A_t - B_t \Phi_t )(\hat{X}_t - X_t) dt + B_t\tonde{ \Psi_t(\mu_t - \E[\mu_t]) + R_t^{-1}\delta_t - \Theta_t }dt, \: \hat{X}_0 - X_0 = 0,
\end{aligned}
\end{equation*}
so that it holds
\begin{equation}\label{LQ:decomposition}
    \hat{X}_t = X_t + f_t(\mu), \quad 0 \leq t \leq T, \: \prob\text{-a.s.}
\end{equation}
In particular, we note that $f(\mu)$ is $\sigma(\mu)$-measurable.
Then, we have
\begin{equation*}
\begin{aligned}
    \E & [\innprod{M_t \hat{X}_t,\hat{X}_t}] = \E[\innprod{M_t (X_t + f_t(\mu)),X_t + f_t(\mu)}] \\
    & = \E[\innprod{M_tX_t,X_t}] + \E[\innprod{M_t (\mu_t + f_t(\mu)),\mu_t + f_t(\mu)}] - \E[\innprod{M_t \mu_t,\mu_t}],
\end{aligned}
\end{equation*}
where we have used the fact that $X$ satisfies the consistency condition \eqref{LQ:def_cce:cons}.
Therefore, we have
\begin{align*}
    \J' & (\hat{\beta},\mu) = \E\bigg[ \int_0^T \Big(  \frac{1}{2}\innprod{M_t X_t,X_t} +\frac{1}{2}\innprod{M_t (\mu_t + f_t(\mu)),\mu_t + f_t(\mu)} -\frac{1}{2}\innprod{M_t \mu_t,\mu_t} + \innprod{N_t X_t, \E[\mu_t] } \\
    & + \innprod{N_t f_t(\mu), \E[\mu_t] }  + \frac{1}{2}\innprod{G_t \E[\mu_t],\E[\mu_t]} +\innprod{q_t -\Phi_t^\top r_t, X_t } + \innprod{q_t -\Phi_t^\top r_t, f_t(\mu) }  \\
    & +\innprod{ (R_t \Phi_t - S_t)  X_t , \Theta_t } + \innprod{(R_t \Phi_t - S_t)  f_t(\mu) , \Theta_t }
    + \innprod{ R_t \Psi_t \E[\mu_t] , \Theta_t }  -\innprod{\Psi_t^T r_t, \E[\mu_t] } \\
    & +\frac{1}{2}\innprod{R_t\Theta_t,\Theta_t } -\innprod{r_t,\Theta_t} \Big) dt  + \frac{1}{2}\innprod{H X_T, X_T} + \frac{1}{2}\innprod{H (\mu_T + f_T(\mu)), \mu_T + f_T(\mu) } -\frac{1}{2}\innprod{H\mu_T, \mu_T} \\
    & + \innprod{\Tilde{H} X_T, \E[\mu_T]} + \innprod{\Tilde{H} f_T(\mu), \E[\mu_T]} \bigg].
\end{align*}
Since the correlated flow $(\lambda,\mu)$ satisfies the consistency condition \eqref{LQ:def_cce:cons}, and noticing that
\begin{align*}
    \E & [\innprod{N_t \mu_t, \E[\mu_t] - \mu_t} ] =  - \E[\innprod{N_t (\E[\mu_t]-\mu_t), \E[\mu_t] - \mu_t} ], \\
    \E & [\innprod{(R_t\Phi_t-S_t) \mu_t , \Theta_t }] = \innprod{(R_t \Phi_t - S_t) \E[\mu_t] , \Theta_t },
\end{align*}
we obtain
\begin{align*}
    \J' & (\hat{\beta},\mu) - \J'(\lambda,\mu) = \E\bigg[ \int_0^T \Big( \frac{1}{2}\innprod{M_t (\mu_t + f_t(\mu)),\mu_t + f_t(\mu)} -\frac{1}{2}\innprod{M_t \mu_t,\mu_t}  \\
    & - \innprod{N_t (\E[\mu_t]-\mu_t), \E[\mu_t] - \mu_t}  + \innprod{N_t f_t(\mu), \E[\mu_t] }  +\frac{1}{2}\innprod{G_t\E[\mu_t],\E[\mu_t]} - \frac{1}{2}\innprod{G_t \mu_t,\mu_t} \\
    & + \innprod{q_t - \Phi_t^\top r_t, f_t(\mu) }  + \innprod{(R_t(\Phi_t+\Psi_t) - S_t ) \E[\mu_t] , \Theta_t }-\innprod{(R_t(\Phi_t+\Psi_t) - S_t ) \mu_t , R_t^{-1}\delta_t } \\
    & + \innprod{ (R_t\Phi_t - S_t) f_t(\mu) , \Theta_t }    +\frac{1}{2} \innprod{R_t\Theta_t,\Theta_t } -\frac{1}{2} \innprod{R_t^{-1}\delta_t ,\delta_t } -\innprod{r_t, \Theta_t - R_t^{-1}\delta_t } \Big) dt \\
    & + \frac{1}{2}\innprod{H (\mu_T + f_T(\mu)),\mu_T + f_T(\mu)} -\frac{1}{2}\innprod{H \mu_T,\mu_T}  \\
    & - \innprod{\tilde H (\E[\mu_T]-\mu_T), \E[\mu_T] - \mu_T}  + \innprod{\Tilde{H} f_T(\mu),\E[\mu_T]} \bigg].
\end{align*}
Therefore, the correlated flow $(\lambda,\mu)$ defined by \eqref{LQ:corr_flow} is a correlated flow if and only if the RHS above is non-negative.
By rearranging the terms and using the equalities
\begin{equation}\label{LQ_general:optimality:final_identities}
R_t \Phi_t - S_t = B_t ^\top \phi_t, \quad R_t \Psi_t = B_t ^\top \psi_t,    
\end{equation}
we get condition \eqref{LQ:cce:final_optimality}.
\end{proof}

The condition for a correlated flow of class $\mathcal G$ to be a mean field CCE is now reduced to an optimality condition which only depends on the population average state and the correlating device of the mediator, i.e. on the joint law of $(\delta,\mu)$. Even though the inequality looks quite long, it can become very tractable and easy to interpret when one specifies some class of dynamics for $\mu$ as done in Section \ref{sec:AG}.

\section{Comparison with MFC solution and mean field NE}\label{sec:comparison}
In this section, we analyze the relationship between mean field CCEs, mean field NEs and MFC solutions. In more detail, we prove the following results:
\begin{itemize}
    \item We compute the MFC solution $\hat{\alpha}^{MFC}$ and we show that it is unique and optimal in the broader class of controls $\mathbb H^2(\mathbb F)$.
    This is accomplished in Proposition \ref{MKV:prop_mkv_solution} and Lemma \ref{MFC:strict_optimality}.
    Then, we show that no mean field CCE can outperform the payoff of the MFC solution. Moreover, if the MFC solution is not a mean field NE, we establish that the MFC payoff is unattainable by a mean field CCE. This is accomplished in Theorem \ref{thm:no_outperf_over_mfc}.
    \item As for mean field NE, we show in Proposition \ref{LQ:MFG:prop_mfg_solution} that there exists a unique mean field NE in our setting.
    Then, we show that for a correlated flow $(\lambda,\mu)$ to be a mean field CCE different from the mean field NE, it is necessary that the flow of moments $\mu$ is stochastic, which is the content of Theorem \ref{LQ_general:no_deterministic_cce}.
    \item Finally, Theorem \ref{LQ_general:cce_mfg_comparison} gives a condition so that a mean field CCE $(\lambda,\mu) \in \mathcal G$  yields a higher payoff than the mean field NE.
\end{itemize}
We remark that the results in the first two points above are fully general, in the sense that they do not restrict to correlated flows in the class $\mathcal{G}$ defined by \eqref{LQ:corr_flow}, while the condition on a mean field CCE $(\lambda,\mu)$ to outperform the payoff of the mean field NE is provided only for correlated flows in $\mathcal{G}$.

We recall here for reader's convenience the definitions of both mean field NE and MFC solution.

\begin{definition}\label{MFG:definition}
We say that a pair $(\hat{\alpha},\hat{m}) \in \mathbb H^2(\mathbb F^1) \times \mathcal{C}([0,T];\R^d)$ is a mean field Nash equilibrium if the following properties hold:
\begin{enumerate}[label=(\roman*)]
\item Optimality: $\hat{\alpha}$ maximizes $\J(\cdot,\hat{m})$ over $\mathbb H^2(\mathbb F^1)$, i.e.,
\begin{equation}\label{MFG:definition:opt}
    \J(\hat{\alpha},\hat{m}) = \max_{\beta \in \mathbb H^2(\mathbb F^1)} \J(\beta,\hat{m}).
\end{equation}
\item Consistency: let $X^{NE}=(X^{NE}_t)_{t \in [0,T]}$ be the solution to equation \eqref{LQ_general:dynamics} with the control process $\hat{\alpha}$.
For every time $t \in [0,T]$, $\hat{m}_t$ equals the expectation of $X^{NE}_t$, i.e.,
\begin{equation}\label{MFG:definition:cons}
    \hat{m}_t=\E[X^{NE}_t], \quad \forall t \in [0,T].
\end{equation}
\end{enumerate}
\end{definition}

\begin{definition}\label{MFC:definition}
For any $\beta \in \mathbb H^2(\mathbb F^1)$, let $X^{\beta}$ be the solution of equation \eqref{LQ_general:dynamics} with $\beta$ instead of $\lambda$.
Denote by $\E[X^{\beta}]=(\E[X^{\beta}_t])_{t \in [0,T]}$ the corresponding flow of first order moments.
We say that a strategy $\hat{\alpha}^{MFC}$ is a MFC solution, if
\begin{equation}\label{MFC:definition:optimality}
    \J(\hat{\alpha}^{MFC},\E[X^{\hat{\alpha}^{MFC}}]) = \max_{\beta \in \mathbb H^2(\mathbb F^1)} \J(\beta,\E[X^{\beta}]).
\end{equation}
\end{definition}

\subsection{Comparison with MFC solution}
In this subsection we compare the expected payoffs of mean field CCEs and the MFC solution.
In our setting, there exists a unique MFC solution.
Since computations are very standard, we postpone them to the \nameref{sec:appendix}.

\begin{prop}\label{MKV:prop_mkv_solution}
Let $\phi^{MFC}$ and $\theta^{MFC}$ be the solutions of the following equations:
\begin{equation}\label{MFC:mean_riccati}
\left \{ \begin{aligned}
    & \dot{\phi}^{MFC}_t + \phi^{MFC}_t A_t + A^\top_t \phi^{MFC}_t + (Q_t + 2\Tilde{Q}_t + \Bar{Q}_t) -(\phi^{MFC}_t B_t  + S_t^\top )R_t^{-1} (B_t^\top \phi^{MFC}_t + S_t)  = 0, \\
    & \phi^{MFC}_T = H + 2\Tilde{H} + \Bar{H}, \\
    & \dot{\theta}^{MFC}_t + A_t^\top \theta^{MFC}_t + q_t - L_t - (\phi^{MFC}_t B_t  + S_t^\top )R_t^{-1} ( B_t^\top \theta^{MFC}_t + r_t ) = 0, \\
    & \theta^{MFC}_T = 0.
\end{aligned} \right.
\end{equation}
Define $\Bar{A}$ and $\Bar{B}$ as
\begin{equation}\label{MFC:flow_of_moments:coefficients}
\begin{aligned}
    & \Bar{A}_t = A_t - B_tR_t^{-1}B_t \phi^{MFC}_t -B_tR_t^{-1}S_t,  && \Bar{B}_t =B_tR_t^{-1}(B_t^\top\theta^{MFC}_t + r_t). 
\end{aligned}
\end{equation}
Let $\bar{\psi}$ and $\bar{\theta}$ be the solutions of the following equations:
\begin{equation}\label{MFC:riccati_control}
\left\{ \begin{aligned}
    & \dot{\bar{\psi}}_t + \Bar{A}_t^\top\bar{\psi}_t + A_t^\top \bar{\psi}_t + (\Bar{Q}_t+2\Tilde{Q}_t) - (\phi_t B_t +S^\top _t)  R_t^{-1} B_t^\top\bar{\psi}_t  = 0, && \bar{\psi}_T = \Bar{H} +2\Tilde{H}, \\
    & \dot{\bar{\theta}}_t -\bar{\psi}_t\Bar{B}_t + A_t^\top\bar{\theta}_t + q_t  - (\phi_t B_t + S_t ^\top) R_t^{-1} (B_t^\top\bar{\theta}_t + r_t) = 0, && \bar{\theta}_T = 0. \\
\end{aligned} \right.
\end{equation}
Let $\phi$ be the solution of the matrix Riccati equation in \eqref{LQ:riccati_eq}.
There exists a unique MFC solution $\hat{\alpha}^{MFC}$, which is given by
\begin{subequations}\label{MFC:solution}
\begin{align}
& \hat{\alpha}^{MFC}_t = -R_t^{-1} (( B_t ^\top \phi_t + S_t)X^{MFC}_t +  B^\top \bar{\psi}_t \bar{x}^{MFC}_t +(B_t ^\top\bar{\theta}_t + r_t)), \label{MFC:solution:control} \\
& \dot{\bar{x}}^{MFC}_t = (A_t - B_tR_t^{-1}B_t \phi^{MFC}_t -B_tR_t^{-1}S_t )\bar{x}^{MFC}_t -B_tR_t^{-1}(B_t^\top\theta^{MFC}_t + r_t), \quad \bar{x}^{MFC}_0 = \nu_1, \label{MFC:solution:flow_of_moments}
\end{align}
\end{subequations}
where $X^{MFC}$ is the solution of
\begin{equation}\label{MFC:solution:state_process}
\left \{ \begin{aligned}
& dX^{MFC}_t = ( (A_t -R_t^{-1} ( B_t ^\top \phi_t + S_t) )X^{MFC}_t   -R_t^{-1}B_t ^\top \bar{\psi}_t \bar{x}^{MFC}_t -R_t^{-1}(B_t ^\top\bar{\theta}_t + r_t) )dt + \sigma_tdW_t,  \\
& X_0 = \xi. \end{aligned} \right. 
\end{equation}
In particular, it holds $\bar{x}^{MFC}_t=\E[X^{MFC}_t]$ for every $t \in [0,T]$.
\end{prop}

Showing that no mean field CCE can outperform the payoff of the MFC solution requires first to show that the MFC solution is actually optimal over the larger control set $\mathbb H^2(\mathbb F)$, as it is done in the following preliminary lemma:
\begin{lemma}\label{MFC:strict_optimality}
Let $\hat{\alpha}^{MFC}$ be the solution of the MFC problem.
Then, for any $\beta$ in $\mathbb H^2(\mathbb F)$, $\beta \neq \hat{\alpha}$, it holds
\begin{equation}
    \J(\hat{\alpha}^{MFC},\bar{x}^{MFC}) > \J(\beta,\E[X^\beta]).
\end{equation}
\end{lemma}
\begin{proof}
To ease the notation, we set
\begin{equation}\label{MFC:MFC_payoff}
    \J^{MFC}(\alpha) :=\J(\alpha,\E[X^{\alpha}]),
\end{equation}
where the process $X$ has dynamics given by \eqref{LQ_general:dynamics}, for any $\alpha \in \mathbb H^2(\mathbb F)$.
We observe that
\begin{align*}
&\E\quadre{ \frac{1}{2}\innprod{\bar Q_t \E[X_t], \E[X_t] } + \innprod{\Tilde{Q}_t X_t, \E[X_t] }+ \frac{1}{2}\innprod{ Q_t X_t, X_t }  + \frac{1}{2}\innprod{R_t\alpha_t,\alpha_t} + \innprod{S_t X_t, \alpha_t } }\\
& = \E\quadre{ \frac{1}{2} \innprod{(\bar Q_t +2\Tilde{Q}_t) \E[X_t], \E[X_t] } + \frac{1}{2}\innprod{ Q_t X_t, X_t }   + \frac{1}{2}\innprod{R_t\alpha_t,\alpha_t} + \innprod{S_t X_t, \alpha_t } } \\
& = \E\left[ \frac{1}{2}\innprod{(\bar Q_t +2\Tilde{Q}_t +Q_t) \E[X_t], \E[X_t] } + \frac{1}{2}\innprod{ Q_t (X_t-\E[X_t]), (X_t-\E[X_t]) }  \right. \\
& \left. + \frac{1}{2}\innprod{R_t\alpha_t,\alpha_t} + \innprod{S_t (X_t-\E[X_t]), \alpha_t } 
+ \innprod{S_t\E[X_t], \alpha_t } \right].
\end{align*}
By Assumptions \ref{LQ:assumptions}, this equality implies that the running payoff in the functional $\J^{MFC}$ is strictly concave jointly in $\E[X_t]$, $X_t - \E[X_t]$ and $\alpha_t$, for every $t \in [0,T]$.
Since $\J^{MFC}$ is also upper semi-continuous, this implies that the maximum exists and that it is unique over the broader class $\mathbb H^2(\mathbb F)$.

\medskip
\noindent We are left to show that the maximum point is indeed $\hat{\alpha}^{MFC}$.
For the sake of clarity, we set
\begin{equation}\label{MKV:running_final_costs}
\begin{aligned}
f(t,x,m,a) & = \innprod{L_t,m} -\frac{1}{2}\innprod{\bar Q_t m, m } -  \frac{1}{2}\innprod{ Q_t x, x }  - \innprod{\Tilde{Q}_t x, m } -\innprod{q_t,x} \\
& - \frac{1}{2}\innprod{R_ta,a} - \innprod{S_t x, a } - \innprod{r_t,a}, \\
g(x,m) & = -\left( \frac{1}{2}\innprod{ \bar H m, m } + \frac{1}{2}\innprod{ H x, x } + \innprod{\Tilde{H} x, m } \right).
\end{aligned}
\end{equation}
Let $\beta$ in $\mathbb H^2(\mathbb F)$.
We define the following process:
\begin{equation}
\Tilde{\beta}_t = \E[ \beta_t \vert \F^1_t ], \quad t \in [0,T].
\end{equation}
Since $\mathbb{F}^1$ satisfies the usual assumptions, $\tilde\beta$ can be taken $\mathbb{F}^1$-progressively measurable (see, e.g., \cite[Section 2]{bremaud1978changes}).
Let $\Tilde{X}=(\Tilde{X}_t)_{t \in [0,T]}$ be the solution of 
\[
d\Tilde{X}_t = (A_t \Tilde{X}_t + B_t\Tilde{\beta}_t)dt + \sigma_t dW_t, \quad \Tilde{X}_0 = \xi.
\]
Then, using the explicit expression for the solution $\tilde X$ of the SDE above, it can be shown by direct computation that
\begin{equation*}
    \Tilde{X}_t = \E[ X^\beta_t \vert \F^1_t ] \quad \prob\text{-a.s.},\quad t \in [0,T].
\end{equation*}
Due to the concave linear quadratic structure of $f$, we have the following:
\begin{align*}
\J^{MFC}(\beta) & =\E\quadre{ \int_0^T f(t,X^\beta_t,\E[X^{\beta}_t],\beta_t) dt + g(X^\beta_T,\E[X^\beta_T])} \\
& = \int_0^T \E\quadre{ \E\quadre{ f(t,X^\beta_t,\E[X^{\beta}_t],\beta_t) \; \vert \; \F^1_t} } dt + \E\quadre{ \E\quadre{ g(X^\beta_T,\E[X^\beta_T]) \; \vert \; \F^1_T} } \\
& \leq \E\quadre{ \int_0^T f(t,\E[X^\beta_t \vert \F^1_t],\E[X^{\beta}_t],\E[\beta_t \vert \F^1_t]) dt + g(\E[X^\beta_T\vert \F^1_T],\E[X^\beta_T])} \\
& = \E\quadre{ \int_0^T f(t,\Tilde{X}_t ,\E[\Tilde{X}_t],\Tilde{\beta}_t) dt + g(\Tilde{X}_T,\E[\Tilde{X}_T])} =  \J^{MFC}(\Tilde{\beta}),
\end{align*}
where we have used the fact that $\E[X^\beta_t]=\E[\Tilde{X}_t]$ for every time $t$.
Since $\Tilde{\beta}$ belongs to $\mathbb H^2(\mathbb F^1)$, Proposition \ref{MKV:prop_mkv_solution} implies 
\[
\J^{MFC}(\beta) \leq \J^{MFC}(\Tilde{\beta}) \leq \J^{MFC}(\hat{\alpha}^{MFC}).
\]
By strict concavity, we deduce that the inequality is strict for any $\beta \neq \hat{\alpha}$.
\end{proof}

In the next theorem we prove that the MFC solution provides an upper bound to the payoffs of any mean field CCEs. Moreover, this upper bound can not be attained unless the MFC solution is a mean field NE. 

\begin{thm}[No outperformance over the MFC solution]\label{thm:no_outperf_over_mfc}
Let $(\lambda,\mu)$ a mean field CCE. Then, the following holds: 
\begin{enumerate}[label=(\roman*)]    
\item If $\J(\lambda,\mu) \geq \J(\hat \alpha^{MFC},\bar x^{MFC})$, then $(\lambda,\mu) = (\hat{\alpha}^{MFC},\bar{x}^{MFC})$, so $\J(\lambda,\mu) = \J(\hat \alpha^{MFC},\bar x^{MFC})$; \label{LQ:cce_MFC:prop_no_outperformance}
\item If the MFC solution is not a mean field NE, then $\J(\lambda,\mu) < \J(\hat \alpha^{MFC},\bar x^{MFC})$. In particular, the MFC solution is not a mean field CCE either. \label{item:mfc_not_cce}
\end{enumerate}
\end{thm}

\begin{proof}[Proof of \ref{LQ:cce_MFC:prop_no_outperformance}]
By using the payoff functional $\J^{MFC}$ defined by \eqref{MFC:MFC_payoff}, the payoffs' inequality reads as
\begin{equation}\label{LQ:cce_MFC:absurd}
    \J(\lambda,\mu) \geq \J(\hat \alpha^{MFC},\bar x^{MFC}) = \J^{MFC}(\hat{\alpha}^{MFC}).
\end{equation}
We reformulate the MFC problem weakly, by taking advantages of the results of \cite[Paragraph 6.6]{librone_vol1}.
We define the set $\mathcal{A} \subseteq \mathcal{P}(\R^d \times \mathcal{C}([0,T];\R^d) \times L^2([0,T];\R^k))$ of admissible probability measures in the following way: take any filtered probability space $(\Omega,\F,\mathbb F,\prob)$ satisfying the usual assumptions, equipped with a $d$-dimensional $\mathbb F$-Brownian motion $W$ and an $\F_0$-measurable random variable $\xi$ independent of $W$.
Let $\alpha \in \mathbb H^2(\mathbb F)$, which we regard as random variable taking values in $L^2([0,T];\R^k)$.
Let $X = X^{\alpha}$ be the solution of
\begin{equation}\label{MKV_dynamics}
    dX^{\alpha}_t = (A_t X^{\alpha}_t + B_t\alpha_t) dt + dW_t, \quad X^{\alpha}_0 = \xi.
\end{equation}
Then, a probability measure $P$ belongs to $\mathcal{A}$ if $P = \prob \circ (\xi,X^{\alpha},\alpha)^{-1}$.
For any $P \in \mathcal{A}$, set $\bar{x}_t = \int_{\R^d} y (P \circ x_t^{-1})(dy)$.
By recalling the definitions of $f$ and $g$ in \eqref{MKV:running_final_costs}, define the payoff functional
\begin{equation}\label{MKV:J_weak}
\begin{aligned}
\mathcal{J} & (P) = \int_{\R^d \times \mathcal{C}([0,T],\R^d) \times L^2([0,T];\R^k)} \tonde{ \int_0^T f(t,x_t,\bar x_t,a)dt + g(x_T,\bar x_T) }P(dz,dx,da) \\
& = \E^{\prob}\quadre{ \int_0^T \tonde{ f(t,X_t,\E^{\prob}[X_t],\alpha_t) } dt + g(X_T,\E^{\prob}[X_T]) }.
\end{aligned}
\end{equation}
By \cite[Theorem 6.37]{librone_vol1}, there exists a probability measure $P^*$ in $\mathcal{A}$ so that 
\begin{equation}\label{MKV:weak_optimality}
    \mathcal{J}(P^*) \geq \mathcal{J}(P) \quad \forall P \in \mathcal{A}.
\end{equation}
Let $(\xi,X^{MFC},\hat{\alpha}^{MFC})$ be the MFC solution given by Proposition \ref{MKV:prop_mkv_solution} and let $\hat{P}$ be its law.
Let $(\lambda,\mu)$ be a mean field CCE and $(\xi,X^{\lambda},\lambda)$ be the corresponding initial state, state process and correlated strategy.
We show the following properties:
\begin{enumerate}
    \item \label{MKV:programme:minimum} The maximum point $P^*$ is unique and it is equal to $\hat{P}$.
    \item \label{MKV:programme:disintegration} For every $m$ in the support of $\mu$, there exists a version of the regular conditional probability of $(\xi,\lambda,X^{\lambda})$ given $\mu=m$; if we set $P^m = \prob( (\xi,\lambda,X^{\lambda}) \in \cdot \; \vert \; \mu = m)$, then $P^m$ belongs to $\mathcal{A}$, and it holds
    \[
    \J (\lambda,\mu) = \int_{\mathcal{C}([0,T];\R^d)} \mathcal{J}(P^m) \rho(dm),
    \]
    where $\rho$ denotes the law of $\mu$.

    \item \label{MKV:programme:equality_controls} We use the above equality to show that $\mu = \bar{x}^{MFC}$ $\prob$-a.s. and deduce $\lambda = \hat{\alpha}^{MFC}$ $d\prob \otimes dt$-a.e..
\end{enumerate}
As for point \ref{MKV:programme:minimum}, let $P^*$ be the admissible probability measure that maximizes $\mathcal{J}$.
Let $(\Omega^*,\F^*,\mathbb{F}^*,\prob^*)$, $W^*$, $\xi^*$, $\alpha^*$ and $X^*$ be so that $P^* = \prob^* \circ (\xi^*,\alpha^*,X^*)^{-1}$. By applying Proposition \ref{MKV:prop_mkv_solution} in this probability space, there exists an optimal control $\hat{\beta}$ which maximizes $\J$ over $\mathbb H^2(\mathbb F^*)$.
Since the flow of moments of $X^{\hat{\beta}}$ is still given by \eqref{MFC:solution:flow_of_moments} and \eqref{MFC:solution:state_process} admits a strong solution, we have $\prob^*\circ(\xi^*,X^{\hat{\beta}},\hat{\beta})^{-1}= \prob\circ(\xi, X^{MFC},\hat{\alpha}^{MFC})^{-1} = \hat{P}$.
Therefore, we can conclude that
\begin{align*}
\mathcal{J} & (P^*) = \E^{\prob^*}\quadre{ \int_0^T \tonde{ f(t,X^*_t,\E^{\prob^*}[X^*_t],\alpha^*_t) } dt + g(X^*_T,\E^{\prob^*}[X_T]) } \\
& \geq \E^{\prob^*}\quadre{ \int_0^T \tonde{ f(t,X^{\hat{\beta}}_t,\E^{\prob^*}[X^{\hat{\beta}}_t],\hat{\beta}_t) } dt + g(X^{\hat{\beta}}_T,\E^{\prob^*}[X^{\hat{\beta}}_T]) } = \mathcal{J}(\hat{P}),
\end{align*}
with the inequality being strict if $\beta^* \neq \hat{\beta}$.
This shows point \ref{MKV:programme:minimum}.

\medskip
As for point \ref{MKV:programme:disintegration}, we can suppose without loss of generality that $(\Omega,\F,\prob)$ is a Polish probability space.
We note that the state process $X^\lambda$ is adapted to the filtration generated by $\xi$, $W$ and $\lambda$, which is countably generated.
This implies that there exists a version of the regular conditional probability of $\prob$ given $\mu=m$, that we denote by $\prob^m$.
Since $\xi$ and $W$ are independent of $\mu$, it is straightforward to see that $W$ is a Brownian motion under $\prob^m$ as well, the law of $\xi$ under $\prob^m$ is $\nu$ and that $X^\lambda$ still satisfies equation \eqref{LQ_general:dynamics}.
Let $P^m = \prob^m\circ(\xi,X^\lambda,\lambda)^{-1}$ and observe that $P^m$ belongs to $\mathcal{A}$ for $\rho$-a.e $m$ in $\mathcal{C}([0,T],\R^d)$.
The consistency condition implies that $\E^{\prob^m}[X^\lambda_t] = m_t$ for $\rho$-a.e. $m$, which in turn implies that 
\begin{align*}
\J & (\lambda,\mu)= \E^{\prob}\quadre{ \int_0^T \tonde{ f(t,X^\lambda_t,\mu_t,\lambda_t) } dt + g(X_T,\mu_T) } \\
& = \E^{\prob}\quadre{ \E^{\prob}\quadre{\int_0^T \tonde{ f(t,X^\lambda_t,\mu_t,\lambda_t) } dt + g(X_T,\mu_T) \vert \mu }} \\
& = \int_{\mathcal{C}([0,T];\R^d)}  \E^{\prob^m}\quadre{ \int_0^T \tonde{ f(t,X^\lambda_t,\E^{\prob^m}[X^\lambda_t],\lambda_t) } dt + g(X_T,\E^{\prob^m}[X^\lambda_T]) dt }\rho(dm) \\
& = \int_{\mathcal{C}([0,T],\R^d)}\mathcal{J}(P^m)\rho(dm).
\end{align*}
By \eqref{LQ:cce_MFC:absurd} and \eqref{MKV:weak_optimality}, we have
\[
\int_{\mathcal{C}([0,T],\R^d)}(\mathcal{J}(\hat{P}) - \mathcal{J}(P^m))\rho(dm) \leq 0, \quad \mathcal{J}(\hat{P}) \geq \mathcal{J}(P^m),
\]
which implies $\mathcal{J}(\hat{P}) = \mathcal{J}(P^m)$ for $\rho$-a.e. $m$.
Since $\hat{P}$ is the unique maximizer of 
$\mathcal{J}$ by point \ref{MKV:programme:minimum}, we get $P^m = \hat{P}$ $\rho$-a.e..
In particular, this implies
\[
m_t=\E^{\prob^m}[X^\lambda_t] = \int_{\R^d} y (P^m \circ x_t)^{-1}(dy) = \int_{\R^d}y (\hat{P} \circ x_t)^{-1}(dy) = \E[X^{MFC}_t] = \bar{x}^{MFC}_t \: \text{for $\rho$-a.e. $m$.}
\]
Thus, $\mu$ is a.s. equal to $\bar{x}^{MFC}$, so that the consistency condition \eqref{LQ:def_cce:cons} for the mean field CCE $(\lambda,\mu)$ rewrites as $\bar{x}^{MFC}_t = \E[X^\lambda_t]$.
Therefore, we have
\[
\J(\lambda,\mu) = \E^{\prob}\quadre{ \int_0^T \tonde{ f(t,X^\lambda_t,\E[X^\lambda_t],\lambda_t) } dt + g(X_T,\E[X^\lambda_T]) } = \J^{MFC}(\lambda).
\]
Since, by Lemma \ref{MFC:strict_optimality}, $\hat{\alpha}^{MFC}$ is unique, the previous equality implies that $\lambda$ is equal to $\hat{\alpha}^{MFC}$ $d\prob\otimes dt$-a.e., which concludes the proof.
\end{proof}

\begin{proof}[Proof of \ref{item:mfc_not_cce}]
Let us assume that the MFC solution $(\hat \alpha^{MFC}, \bar x^{MFC})$ is not a mean field NE (see upcoming Definition \ref{MFG:definition}).
By item  \ref{LQ:cce_MFC:prop_no_outperformance} of Theorem \ref{thm:no_outperf_over_mfc}, every mean field CCE yields a lower payoff than the MFC solution; moreover, if there was a mean field CCE yielding the same payoff as the MFC solution, it would be the MFC solution itself.
Therefore, we just need to prove that the  MFC solution is not a mean field CCE.

The pair $(\hat{\alpha}^{MFC},\bar{x}^{MFC})$ is a correlated flow which satisfies the consistency condition in the definition of mean field CCE.
Moreover, since $\hat{\alpha}^{MFC}\in \mathbb H^2(\mathbb F^1)$ and $\bar x^{MFC}$ is deterministic, it satisfies the consistency condition of the definition of the mean field NE as well.
Since by assumption the MFC solution is not a mean field NE, it is the optimality condition \eqref{MFG:definition:opt} in the upcoming definition of mean field NE which is not satisfied. Therefore, there exists $\beta \in \mathbb H^2(\mathbb F^1)$ so that $\J(\beta,\bar x^{MFC}) >  \J(\hat \alpha^{MFC},\bar x^{MFC})$.
Since such $\beta$ is an admissible deviation to the correlated flow $(\hat{\alpha}^{MFC},\bar{x}^{MFC})$, the optimality condition  \eqref{LQ:def_cce:opt} in definition of mean field CCE is not satisfied either. This means that the MFC solution is not a mean field CCE.
\end{proof}

\subsection{Comparison with mean field Nash equilibria}
As for mean field NE, we first show that the only mean field CCE with deterministic flow of moments is the mean field NE itself.
In particular, this implies that randomization of the flow of moments is needed for mean field CCEs to reach higher payoffs than the mean field NE.
Differently from the MFC case, no general outperformance result can be established for mean field CCEs.
Instead, one can derive an outperformance condition for correlated flows in the class $\mathcal G$, in a similar approach as for the optimality condition in subsection \ref{sec:optimality_condition}.

As shown by next proposition, there exists a unique mean field NE.
The proof is a standard application of the Pontryagin maximum principle approach together with the fixed point argument of \cite[Chapter 4]{librone_vol1}.
We include it in the \nameref{sec:appendix} for the sake of completeness.

\begin{prop}\label{LQ:MFG:prop_mfg_solution}
Let $\phi^{NE}$ and $\theta^{NE}$ the solutions of the following system:
\begin{equation}\label{MFG:fixed_point_riccati}
\left \{ \begin{aligned}
    & \dot{\phi}^{NE}_t + \phi^{NE}_t A + A^\top \phi^{NE}_t + (Q + \Tilde{Q}) -(\phi^{NE}_t B  + S^\top) R^{-1} (B^\top \phi^{NE}_t + S)  = 0, \\
    & \phi^{NE}_T = H + \Tilde{H}, \\
    & \dot{\theta}^{NE}_t + A^\top \theta^{NE}_t + q - (\phi^{NE}_t B  + S^\top) R^{-1} ( B^\top \theta^{NE}_t + r_t ) = 0, \\
    & \theta^{NE}_T = 0.
\end{aligned} \right.
\end{equation}
Define $\hat{A}$ and $\hat{B}$ as
\begin{equation}\label{MFG:flow_of_moments:coefficients}
\begin{aligned}
    & \hat{A}_t = A_t - B_tR_t^{-1}B_t \phi^{NE}_t -B_tR_t^{-1}S,  && \hat{B}_t =B_tR_t^{-1}(B_t^\top\theta^{NE}_t + r). 
\end{aligned}
\end{equation}
Let $\theta^{\hat{m}}$ be the solution of the following equation:
\begin{equation}\label{MFG:riccati_control}
    \dot{\theta}^{\hat{m}}_t -\psi_t\frac{d\hat{m}_t}{dt} + A_t^\top\theta^{\hat{m}}_t + q_t  - (\phi_t B_t + S^\top)R_t^{-1} (B_t^\top\theta^{\hat{m}}_t + r_t) = 0, \quad  \theta^{\hat{m}}_T = 0.
\end{equation}
Let $\phi$ be the solution of the matrix Riccati equation in \eqref{LQ:riccati_eq}.
There exists a unique mean field NE $(\hat{\alpha},\hat{m})$, which is given by
\begin{subequations}\label{MFG:solution}
\begin{align}
    & \dot{\hat{m}}_t = (A_t - B_tR_t^{-1}B_t \phi^{NE}_t -B_tR_t^{-1}S )\hat{m}_t -B_tR_t^{-1}(B_t^\top\theta^{NE}_t + r), \quad \hat{m}_0 = \nu_1, \label{MFG:solution:flow_of_moments} \\
    & \hat{\alpha}_t = -R_t^{-1} (( B^\top \phi_t + S)X^{NE}_t +  B^\top \psi_t \hat{m}_t + B^\top \theta^{\hat{m}}_t + r_t), \label{MFG:solution:control}
\end{align}
\end{subequations}
where $X^{NE}$ is the solution of equation \eqref{LQ_general:dynamics} with the control $\hat{\alpha}$.
\end{prop}

We observe that, by definition, a mean field NE is a mean field CCE with deterministic flow of measures $\hat{m}$.
The converse is true as well, as shown by the following Theorem:

\begin{thm}\label{LQ_general:no_deterministic_cce}
Let $(\lambda, \mu)$ be a mean field CCE with deterministic $\mu$. Then, $(\lambda, \mu)$ is  the mean field NE.
\end{thm}

\begin{proof}
We start by observing that, by using the same concavity and projections arguments as in the proof of Lemma \ref{MFC:strict_optimality}, we have
\[
\J(\hat{\alpha},\hat{m}) > \J(\beta,\hat{m}) \quad \forall \beta \in \mathbb H^2(\mathbb F), \; \beta \neq \hat{\alpha}.
\]
Let $(\lambda,\mu)$ be a coarse correlated equilibrium with deterministic flow of moments $\mu$.
Then, the consistency condition \eqref{LQ:def_cce:cons} becomes $\mu_t = \E[X^\lambda_t]$ for every time $t$.
By optimality, it holds $\J(\lambda,\mu) \geq \J(\beta,\mu)$ for every $\beta \in \mathbb H^2(\mathbb F^1)$.
By reasoning as in the proof of Lemma \ref{MFC:strict_optimality}, there exists a strategy $\Tilde{\lambda} \in \mathbb H^2(\mathbb F^1)$ so that
\[
\Tilde{\lambda}_t = \E[\lambda_t \vert \F^1_t], \quad X^{\Tilde{\lambda}}_t = \E[ X^\lambda_t \vert \F^1_t ], \quad \prob\text{-a.s.}, \: \forall \; t \in [0,T],
\]
where $X^{\Tilde{\lambda}}$ is the solution of equation \eqref{LQ_general:dynamics} corresponding to the strategy $\Tilde{\lambda}$.
Since $\mu$ is deterministic, by exploiting the convex linear quadratic structure of the payoff functional $\J'$, we have
\begin{equation}\label{LQ_general:no_deterministic_cce:inequalities}
    \J(\beta,\mu) \leq \J(\lambda,\mu) \leq \J(\Tilde{\lambda},\mu), \quad \forall \: \beta \in \mathbb H^2(\mathbb F^1). 
\end{equation}
Since $\mu$ is deterministic by assumption, the consistency condition holds true for the correlated flow $(\Tilde{\lambda},\mu)$ as well, so that \eqref{LQ_general:no_deterministic_cce:inequalities} implies that $(\Tilde{\lambda},\mu)$ is itself a mean field NE.
By uniqueness of the mean field NE, we deduce $(\Tilde{\lambda},\mu) = (\hat{\alpha},\hat{m})$, so that in particular $\mu=\hat{m}$ $\prob$-a.s..
Since $\hat{\alpha}$ is the unique maximizer of $\J(\cdot,\hat{m})$ over $\mathbb H^2(\mathbb F)$, we deduce $\J(\hat{\alpha},\hat{m}) \geq \J(\lambda,\hat{m})$.
Since $(\lambda,\mu)$ is a mean field CCE by assumption, it holds $\J(\hat{\alpha},\hat{m}) \leq \J(\lambda,\hat{m})$, which, by uniqueness, implies that $\lambda = \hat{\alpha}$ $dt\otimes\prob$-a.e. as well.
\end{proof}

Finally, consider again correlated flows $(\lambda,\mu)$ in the class $\mathcal{G}$. By using their specific structure as described in \eqref{LQ:corr_flow}, we are able to provide a condition under which they yield a higher payoff than the mean field NE.
\begin{thm}\label{LQ_general:cce_mfg_comparison}
Let $\theta^{\hat{m}}$ be the solution of \eqref{MFG:riccati_control}. Set
\begin{equation}\label{LQ_general:cce_mfg:auxiliary_theta}
\begin{aligned}
    & \Theta^{\hat{m}}_t = R_t^{-1}(B^\top\theta^{\hat{m}}_t + r_t).
\end{aligned}
\end{equation}
Let $(\lambda,\mu) \in \mathcal G$ corresponding to some $\delta \in \mathbb H^2(\F_0)$.
Then, $\J(\lambda,\mu)$ is higher than the payoff $\J(\hat{\alpha},\hat{m})$ given by the mean field NE if and only if the following inequality is satisfied:
\begin{equation}\label{LQ:cce_mfg_outperformance:final}
\begin{aligned}
\int_0^T & \bigg( \frac{1}{2}(\innprod{\Bar{Q} \hat{m}_t, \hat{m}_t} - \E[\innprod{ \Bar{Q} \mu_t, \mu_t  }]) + \frac{1}{2} ( \innprod{M_t \hat{m}_t,\hat{m}_t} - \E[\innprod{M_t \mu_t,\mu_t}]) \\
& + \frac{1}{2}(\innprod{G_t \hat{m}_t,\hat{m}_t} - \E[\innprod{G_t \mu_t,\mu_t}])  + \innprod{N_t \hat{m}_t, \hat{m}_t } - \E[\innprod{N_t \mu_t, \mu_t }] \\
&  + \innprod{L_t - q_t +R_t^{-1}(B^\top_t(\phi_t+\psi_t)+S_t)^\top r_t,  \E[\mu_t] -\hat{m}_t } + \innprod{B_t^\top\phi_t\hat{m}_t, \Theta^{\hat{m}}_t  }  \\
& - \E[\innprod{B_t^\top\phi_t\mu_t,R_t ^{-1}\delta_t  }]  - \E[\innprod{B_t^\top\psi_t \mu_t,R_t ^{-1}\delta_t }]  \\
& + \innprod{ B_t^\top\psi_t\hat{m}_t,\Theta^{\hat{m}}_t }  + \frac{1}{2}(\innprod{R_t\Theta^{\hat{m}}_t,\Theta^{\hat{m}}_t } - \E[\innprod{R^{-1}_t\delta_t ,\delta_t }]) - \innprod{r_t, \Theta^{\hat{m}}_t - R_t ^{-1}\E[\delta_t]} \bigg)dt \\
& +\frac{1}{2}(\innprod{\Bar{H}\hat{m}_T,\hat{m}_T} -\E[\innprod{\Bar{H}\mu_T,\mu_T}])  +\frac{1}{2}(\innprod{H\hat{m}_T, \hat{m}_T} - \E[\innprod{H \mu_T, \mu_T }]) \\
& + \innprod{\Tilde{H}\hat{m}_T, \hat{m}_T} - \E[\innprod{\Tilde{H} \mu_T, \mu_T }] \geq 0.
\end{aligned}
\end{equation}
\end{thm}

\medskip
The proof is similar to the one of Theorem \ref{LQ:prop:optimality_condition}. For the sake of completeness, we include it in the \nameref{sec:appendix}.
We observe that, although the inequality \eqref{LQ:cce_mfg_outperformance:final} is not easy to interpret, it involves only the law of $\mu$ and its associated $\delta$.
Moreover, it can be verified separately from the optimality condition \eqref{LQ:cce:final_optimality}, giving some room for mean field CCEs to outperform the mean field NE payoff.
This will be accomplished for the abatement game in Section \ref{sec:AG}.

\section{Application to an emission abatement game}\label{sec:AG}

In this section we consider an emission abatement game inspired by environmental economics literature on international environmental agreements, in line with the very popular model of \cite{barrett1994self}. Previous section's findings allow us to exhibit a simple class of coarse correlated equilibria which (highly) outperforms the mean field NE in this game. 

The emission abatement game has the following payoff and dynamics of the representative player state:
\begin{subequations}\label{ag:setting}
\begin{align}
\J(\alpha,\mu) & = \E\quadre{\int_0^T\tonde{ a\mu_t -\frac{b}{2}\mu_t^2 -\frac{1}{2}\alpha_t^2 -\frac{\eps}{2}(\mu_t - X_t)^2 }dt },\label{ag:payoff}\\
dX_t & = \alpha_tdt+dW_t, \quad X_0 = \xi,\label{ag:state_var}
\end{align}
\end{subequations}
with $a,b$ non-negative constants and $\eps >0$. The strategy $\alpha_t$ represents the abatement rate of the player at time $t$, while $X_t$ models the cumulated abatement over the interval $[0,t]$. 

We translate a slightly modified version of the abatement game \cite{barrett1994self} into a dynamic stochastic mean field game. We follow the $N$-player formulation of \cite{DokkaTrivikram2022Edia} by considering symmetric players, and a normalization of the number of players is implicitly added by replacing the sum of abatement efforts by the flow of moments $\mu$.
We also add the last term in $\eps$, inspired by further developments of this model in the literature (see \cite{gruning2010can}), which can be interpreted as a reputational cost. It appears to be necessary when one wants mean field CCEs outperforming the mean field NE at the mean field limit. Indeed, when $\eps = 0$, there exists only a unique mean field CCE, corresponding to the mean field NE.
This is straightforward by direct computations and can be also deduced from Proposition \ref{ag:cor:optimality_condition} (see upcoming Remark \ref{ag:rem:no_mean_field_interaction}).

Following \cite{barrett1994self}, the other terms of the payoff can be interpreted as follows. The term $a\mu_t - \frac{b}{2} \mu_t^2$, which depends solely on the mean field component $\mu$, is the ``abatement benefit''. It represents the individual benefit of global climate change mitigation allowed by aggregate abatement efforts, with a decreasing marginal benefit. The quadratic term in the control, i.e. $- \frac{1}{2}\alpha_t^2$, is an ``abatement cost'' that the representative country privately pays for its abatement effort.

We do not claim that a mean field approximation of the abatement game of \cite{barrett1994self} is a right way to approach the problem of international environmental agreements economically. We rather use this payoff functional as a toy example that allows us to illustrate very efficiently the interest of mean field CCEs in a context of common good, and to contribute to the findings of \cite{DokkaTrivikram2022Edia}.

\begin{remark}
Going from static to dynamic games also induces some additional assumptions that were not included in reference models \cite{barrett1994self,DokkaTrivikram2022Edia}.
We chose to  represent the ``abatement benefit'' as a running payoff rather than a terminal one, considering that environmental objectives are not only to reach a given level of emissions at a terminal time, but also to abate as much as possible, as early as possible.
\end{remark}

\subsection{Translation and interpretation of findings in the abatement game}
In this subsection we apply the theory developed in the previous section to compute mean field CCEs in the abatement game.
In the next subsection, we will make a step further and exhibit a simple but interesting subclass of correlated flows $(\lambda,\mu)$ which verify both the optimality inequality \eqref{LQ:cce:final_optimality} and the NE outperformance inequality \eqref{LQ:cce_mfg_outperformance:final}.

\medskip
We use the setting of Section \ref{sec:setting} with $d=k=1$. 
The parameters are given by
\begin{equation}\label{ag:param}
A_t = 0,\; B_t = 1,\; \sigma_t = 1,\;  L_t = a,\; \bar Q_t = b + \eps,  \; Q_t = \eps,\; \tilde Q_t = -\eps,\; R_t = 1,\quad  \forall t\in [0,T],
\end{equation}
and remaining parameters equal 0. According to Proposition \ref{LQ_general:prop_best_dev} with the abatement game parameters as in \eqref{ag:param}, for a given correlated flow $(\lambda,\mu)$,  the best deviating strategy and the corresponding state process are given by
\begin{equation}\label{ag:best_deviation}
\begin{aligned}
    & \hat{\beta}_t = \phi_t(\E[\mu_t]-\hat{X}_t)-\theta_t, \\
    & d\hat{X}_t = \hat{\beta}_t dt + dW_t, \quad \hat{X}_0 = \xi,
\end{aligned}
\end{equation}
with $\phi$ and $\theta$ satisfying equations
\begin{equation}\label{ag:riccati_eq}
\left\{ \begin{aligned}
    & \dot\phi_t  + \epsilon -\phi_t^2 = 0, && \phi_T = 0, \\
    & \dot\theta_t - \phi_t\left(\theta_t+\frac{d\E[\mu_t]}{dt} \right)= 0, && \theta_T = 0.
\end{aligned} \right.
\end{equation}
Note that $\psi$ does not appear as in this case $\psi=-\phi$. We stress that, as only unilateral deviation is allowed, the deviating player can not act on the abatement benefit, and therefore does not consider $a$ and $b$ in her optimal strategy. 

The family $\mathcal G$ of correlated flows defined by \eqref{LQ:corr_flow} is composed of any correlated flow $(\lambda,\mu)$ so that:
\begin{equation*}
\begin{aligned}
    & \lambda_t =  \phi_t(\mu_t-X_t) - \delta_t, \\
    & \dot{\mu}_t =  - \delta_t, \quad \mu_0 = \nu_1,
\end{aligned}
\end{equation*}
for some $\delta \in \mathbb H^2(\F_0)$ and where $X$ solves $$dX_t = \lambda_t dt + dW_t,\quad X_0 = \xi.$$
In particular, we note that it holds $\delta_t =-\frac{d\mu}{dt}$.
Therefore, in this model, the class $\mathcal G$ is composed of correlated flows $(\lambda,\mu)$ verifying
\begin{equation}\label{ag:corr_flow}
\begin{aligned}
    \lambda_t =  \phi_t(\mu_t-X_t) + \frac{d\mu}{dt}, \quad \E\quadre{\int_0^T\tonde{\frac{d\mu}{dt}}^2dt} < \infty.
\end{aligned}
\end{equation}
As the correlated strategy depends on $\mu$ itself, we remark that the state variable becomes actually mean-reverting. The extra term $\frac{d\mu}{dt}$ allows the state of the representative player following the suggested strategy to satisfy the consistency condition by following the suggested variations of $\mu$. 

As shown by the next proposition, in the abatement game, the optimality condition only depends on the law of $\mu$, the reputational cost parameter $\eps$, and the final time horizon $T$.

\begin{prop}[Optimality condition for the abatement game]\label{ag:cor:optimality_condition}
Let $(\lambda,\mu)$ be a correlated flow in $\mathcal G$. 
Let $f(\mu)=(f_t(\mu))_{t \in [0,T]}$ be given by
\begin{equation}\label{ag:f_mu}
\left\{ \begin{aligned}
    & \dot{f}_t(\mu) =  -\tonde{\phi_t\tonde{f_t(\mu)  +\mu_t - \E[\mu_t]} +\frac{d\mu}{dt} + \theta_t}, \quad 0 \leq t \leq T, \\
    & f_0(\mu) = 0.
\end{aligned} \right.
\end{equation}
Then, $(\lambda,\mu)$ is a mean field CCE if and only if the following condition is satisfied:
\begin{equation}\label{ag:optim_condition}
    \int_0^T\E\quadre{ \tonde{\frac{d\mu}{dt}}^2  }dt
    \leq \int_0^T  \E\Big[  (\phi_tf_t(\mu)+\theta_t)^2 + \phi_t^2 (\mu_t-\E[\mu_t]+f_t(\mu))^2 +(\eps -\phi_t^2)f_t(\mu)^2   \Big]dt.
\end{equation}
\end{prop}
\begin{proof}
Referring to \eqref{LQ:fb_coefficients} and \eqref{LQ:auxiliary_coefficients}, the auxiliary functions for the abatement game are as follows:
\begin{equation}\label{ag:auxiliary_coefficients}
\begin{aligned}
     & \Phi_t = \phi_t, && \Psi_t = -\phi_t, &&& \Theta_t = \theta_t,\\
     & M_t = \eps + \phi_t^2, && N_t = -\eps - \phi_t^2, &&& G_t = \phi_t^2.
\end{aligned}
\end{equation}
This implies that $f(\mu)$ given in \eqref{LQ:auxiliary_f} takes the form of equation \eqref{ag:f_mu}, recalling that $\delta_t = -\sfrac{d\mu}{dt}$ by \eqref{ag:corr_flow}.
After a few computations, we get that the optimality condition \eqref{LQ:cce:final_optimality} rewrites as 
\begin{align*}
    \int_0^T\E\quadre{ \delta^2_t  }dt
    \leq & \int_0^T  \E\Big[  (\phi_tf_t(\mu)+\theta_t)^2 + \phi_t^2(\mu_t-\E[\mu_t]+f_t(\mu))^2 +(\eps -\phi_t^2)f_t^2(\mu) \\
    & + 2\eps(\mu_t-\E[\mu_t])(\mu_t - \E[\mu_t] + f_t(\mu))    \Big]dt,
\end{align*}
using that 
\[
\E[\mu_t^2-\E[\mu_t]^2] = \E[(\mu_t-\E[\mu_t])^2].
\]
Since $f_t(\mu) = \hat{X}_t - X_t$ $\prob$-a.s., for every time $t$ by \eqref{LQ:decomposition} and $f(\mu)$ is $\sigma(\mu)$-measurable by definition, we have 
\[
f_t(\mu) = \E[f_t(\mu) \vert \mu ]= \E[\hat{X}_t - X_t  \vert \mu ] =  \E[\hat{X}_t] - \mu_t
\]
where we used the consistency condition \eqref{LQ:def_cce:cons} and the fact that $\hat{X}$ and $\mu$ are independent.
This implies that
\begin{align*}
    \E & \quadre{ (\mu_t-\E[\mu_t])(\mu_t - \E[\mu_t] + f_t(\mu)) } = \E\quadre{ \mu_t-\E[\mu_t] }(- \E[\mu_t] +\E[\hat{X}_t] ) = 0.
\end{align*}
Therefore, $(\lambda,\mu)$ is a mean field CCE if and only if condition \eqref{ag:optim_condition} is satisfied.
\end{proof}

By Proposition \ref{LQ:MFG:prop_mfg_solution}, there exists a unique mean field NE $(\hat{\alpha},\hat{m})$, which is given by
\begin{subequations}\label{AG:MFG_solution}
\begin{align}
    & \hat{\alpha}_t = \phi_t(\hat{m}_t - X_t^{NE})   , \label{AG:MFG_solution:control} \\ 
    & \hat{m}_t = \nu_1, \: \forall t \in [0,T], \label{AG:MFG_solution:flow_of_moments}
\end{align}
\end{subequations}
since we have $\phi^{NE} = \theta^{NE} = 0$ in \eqref{MFG:fixed_point_riccati}, which implies $\theta^{\hat{m}} = 0$ in \eqref{MFG:riccati_control} as well.

The mean field NE consists, on average, to null abatement, as $\hat m_t$ stays constant equal to its initial value.
This corresponds to a free-riding equilibrium, where everybody does as little as possible, and prefers to take advantage of the others' efforts.
As a result, nobody does anything.

\begin{remark}\label{ag:rem:no_mean_field_interaction}
One can easily see from the optimality condition in equation \eqref{ag:optim_condition} that, if $\eps=0$, the only mean field CCE is the mean field NE. Indeed, in this case $\phi_t = \theta_t = 0$ and $f_t(\mu)  = \sfrac{d\mu}{dt}$, for all $t\in[0,T]$. Hence the right-hand side term in \eqref{ag:optim_condition} is null, forcing $\sfrac{d\mu}{dt}=0$, $t\in[0,T]$. As $\lambda_t = \sfrac{d\mu}{dt}$, we get $\lambda_t = \hat \alpha_t = 0, \, \hat m_t = \nu_1$, which is the mean field NE when $\eps = 0$. This seems consistent with the findings of  \cite{DokkaTrivikram2022Edia}. Indeed, in an equivalent $N$-player static deterministic game without the reputational cost ($\eps = 0$), the authors find that, the more players, the less the payoff-maximising CCE outperforms the payoff of the NE. This probably comes from the fact that, at the mean field limit, there is only one mean field CCE, which is the mean field Nash equilibrium itself. 
\end{remark}

By Proposition \ref{MKV:prop_mkv_solution}, there exists a unique MFC solution $(\hat{\alpha}^{MFC},\bar x_t^{MFC})$ which reads:
\begin{subequations}\label{ag:MFC_solution}
\begin{align}
    &\hat{\alpha}^{MFC}_t = \phi_t(\bar x_t^{MFC}-X_t^{MFC}) - \bar\eta_t\bar x_t^{MFC} - \bar\chi_t,\\
    &\dot{\bar x}_t^{MFC}  =  -\bar\eta_t\bar x_t^{MFC}  - \bar\chi_t, \quad \bar x^{MFC}_0 = \nu_1,
\end{align}
\end{subequations}
with 
\begin{equation}\label{ag:MFC_solution:odes}
\left \{ \begin{aligned}
    & \dot{\bar\eta}_t = \bar\eta^2_t - b, &&\bar\eta_T = 0, \\
    & \dot{\bar\chi}_t = \bar\eta_t\bar\chi_t +a, &&\bar\chi_T = 0.
\end{aligned} \right.
\end{equation}
The MFC solution adds to the mean-reversion two terms which depend on $a$ and $b$, i.e. on the coefficients of the abatement benefit. One can note that the MFC solution and the mean field NE coincide if and only if $a=b=0$. To the contrary, when the ``common good'' aspect of climate is accounted for in the payoff through the abatement benefit, the central planner can reach higher payoffs by preventing any inefficient free-riding behaviour. This gives some room for mean field CCEs to bridge the gap between the free-riding mean field NE and the central planner optimum. 

To find mean field CCEs outperforming the mean field NE, the following condition should be fulfilled.

\begin{prop}[Outperformance condition over mean field NE]\label{ag:outperf_condition}
Let $(\lambda,\mu)$ be a correlated flow in $\mathcal G$. Then,
\begin{equation*}
    \J(\lambda,\mu)\geq \J(\hat \alpha,\hat m) \; \iff \; \E\quadre{ \int_0^T  \Big( a(\mu_t-\hat m_t)- \frac{b}{2}(\mu_t^2-\hat m_t^2)    \Big)dt } \geq  \frac{1}{2}\E\quadre{ \int_0^T \tonde{\frac{d\mu}{dt}}^2dt}.
\end{equation*}
\end{prop}
\begin{proof}
By recalling the identities in \eqref{ag:auxiliary_coefficients}, inequality \eqref{LQ:cce_mfg_outperformance:final} takes the following form:
\begin{equation*}
\begin{aligned}
\J & (\lambda,\mu)-\J(\hat \alpha,\hat m) = \int_0^T \left ( \frac{1}{2}(b+\eps)(\hat{m}^2_t-\E[\mu_t^2]) +\frac{1}{2}(\eps + \phi^2_t)(\hat{m}^2_t-\E[\mu_t^2]) +\frac{1}{2}\phi_t^2(\hat{m}^2_t-\E[\mu_t^2]) \right.\\
& \left. -(\eps + \phi_t^2)(\hat{m}^2_t-\E[\mu_t^2]) + a(\E[\mu_t]-\hat{m}) -\phi_t\E[\mu_t\delta_t] +\phi\E[\delta_t\mu_t] - \frac{1}{2}\E[\delta^2_t]\right) \geq 0
\end{aligned}
\end{equation*}
By rearranging the terms and recalling that $\delta_t = -\sfrac{d\mu}{dt}$, we get the desired inequality.
\end{proof}
The equivalence in Proposition \ref{ag:outperf_condition} clearly illustrates that, when $a=b=0$, the best payoff mean field CCE is actually the mean field NE.
This was also implied by the fact that, when $a=b=0$, the MFC solution is a mean field NE as we already mention above.

\subsection{A tractable class of mean field CCEs}

In this subsection we show that, when $a\neq 0$ or $b\neq 0$, the optimality and outperformance conditions are not empty, and neither is their intersection.
In this case, the MFC solution is distinct from the mean field NE, which implies, according to Theorem \ref{thm:no_outperf_over_mfc}, that the MFC solution is not a mean field CCE, as the required control does not resist any unilateral deviation which tends to a less costly free-riding option.
However, by introducing correlation through correlated flows, one can manage to drive the population at quite high abatement levels, leading to more desirable social outcomes than the one of the mean field NE.

The optimality condition \eqref{ag:optim_condition} is very convenient to use when one focuses on a specific class of dynamics for $\mu$. In this subsection, we consider a subclass $\mathcal G_l \subseteq \mathcal G$ where the flows of moments are linear in time.

More precisely, let $\mathcal G_l$ be the set of all correlated flows $(\lambda, \mu) \in \mathcal G$ such that 
\begin{align}
     \mu_t = \nu_1 + tZ , \quad t\in [0,T],
\end{align}
for some $Z \in L^2(\F_0)$ independent of $\xi$ and $W$. 
Then, for all correlated flows $(\lambda, \mu) \in \mathcal G_l$ we have
$$\lambda_t  = \phi_t(\mu_t - X_t) + Z, \quad t\in [0,T].$$
In the rest of the paper, we will use the notations $z_1:= \E[Z],\; z_2 := \E[Z^2],\; \sigma_z^2 := \mathbb V[Z]$.

\begin{prop}[Optimality condition for $\mathcal G_l$]\label{ag:linear:opt_cond}
Let $(\lambda, \mu)\in \mathcal G_l$. Then $(\lambda, \mu)$ is a mean field CCE if and only if
\begin{equation}\label{ag:eq:linear:opt_cond}
    z_1^2c_M +\sigma_z^2c_V \geq 0
\end{equation}
with
\begin{equation}\label{ag:value_cM_cV}
    c_M = \int_0^T\tonde{\tonde{\phi_tr_t+g_t}^2 + \eps r_t^2 }dt - T , \quad 
    c_V = \int_0^T\tonde{\phi_t^2\tonde{v_t-t}^2+\eps v_t^2 }dt -T,
\end{equation}
and    \begin{equation}\label{ag:linear_flow:auxiliary_functions}
\begin{aligned}
    g_t = \int_t^T\phi_se^{-\int_t^s\phi_udu}ds, \quad  r_t = \int_0^t (1-g_s) e^{-\int_s^t \phi_u du}ds, \quad  v_t = \int_0^t (s\phi_s+1) e^{-\int_s^t \phi_u du}ds.
\end{aligned}
\end{equation}
\end{prop}
\begin{proof}
For any given $(\lambda, \mu) \in \mathcal G_l$ we have $\frac{d\mu}{dt} = Z\; a.s. $ and $\;\E\quadre{\mu_t} = \nu_1 + tz_1 $
so that
\begin{align*}
    &\int_0^T \E\quadre{\tonde{\frac{d\mu}{dt}}^2} dt = Tz_2,\quad\qquad \dot{\theta}_t = \phi_t\tonde{ \theta_t + z_1}, \; \theta_T=0,\\
    &\dot{f}_t(\mu)= -\tonde{ \phi_t(f_t(\mu)+t(Z-z_1)) +Z +\theta_t}  , \quad f_0(\mu) = 0.
\end{align*}
Let us set $$p_t := \int_0^t e^{-\int_s^t \phi_u du}ds.$$
By using $p_t$ and the auxiliary functions defined in \eqref{ag:linear_flow:auxiliary_functions}, $f_t(\mu)$ and $\theta_t$ can be rewritten as follows:
\[
\theta_t = 
- z_1g_t,\qquad f_t(\mu) = -Zp_t-(Z-z_1)(v_t-p_t)+z_1(p_t-r_t).
\]
We compute the different terms appearing in the integral of the right-hand side of the optimality condition:
\begin{align*}
    & \E[(f_t(\mu))^2]= \sigma_z^2v_t^2+z_1^2r_t^2, && \E[(\mu_t-\E[\mu_t]+f_t(\mu))^2] = \sigma_z^2(v_t-t)^2+z_1^2r_t^2,\\
    & \E[(\phi_tf_t(\mu)+\theta_t)^2] = \sigma_z^2\phi_t^2v_t^2 + z_1^2(\phi_tr_t+g_t)^2.
\end{align*}
After summing, simplifying and factorising, the optimality condition becomes an inequality on the moments of $Z$ as follows:
$$Tz_2 \leq z_1^2\int_0^T\tonde{(\phi_tr_t+g_t)^2 + \eps r_t^2}dt +\sigma_z^2\int_0^T\tonde{\phi_t^2\tonde{v_t-t}^2+\eps v_t^2}dt .$$
As $z_2 = \sigma_z^2 + z_1^2$, we get \eqref{ag:eq:linear:opt_cond} and \eqref{ag:value_cM_cV}.
\end{proof}

Thanks to this simple optimality condition, the set $\mathcal G_l$ of mean field CCEs can be easily explored numerically and analytically.
In Figure \ref{fig:running_obj_with_lin_cce} we represent the running expected payoffs (time derivative of the payoff) of a mean field CCE, the MFC solution and the mean field NE as curves,  and their total payoffs as dots.
Figure \ref{fig:running_abat_effort_comp} represents the average of the state variables at each time for the same equilibria. As one can see, the mean field CCE in the figure outperforms the mean field NE in terms of both payoff and abatement levels.
Moreover, Figure \ref{fig:running_abat_effort_comp} shows that this mean field CCE also outperforms the MFC solution in terms of average level of cumulated abatement at the end of the period, i.e. $\E[\mu_T] > \bar x^{MFC}_T$.

Implications of the optimality condition for $\mathcal G_l$ can be further analysed by stating some of its analytical properties.

\begin{figure}[!ht]
\begin{subfigure}{.48\textwidth}
  \centering
\includegraphics[width =1.1\textwidth]{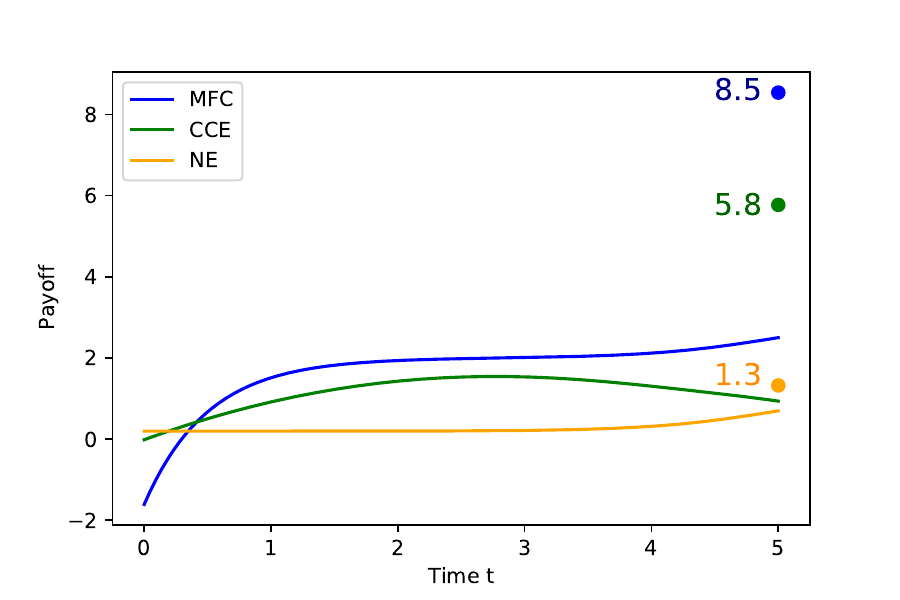}
    \caption{Running expected utility and payoff of a mean field CCE, the MFC solution and the mean field NE.}
    \label{fig:running_obj_with_lin_cce}
\end{subfigure}\hfill
\begin{subfigure}{.48\textwidth}
  \centering
    \includegraphics[width =1.1\textwidth]{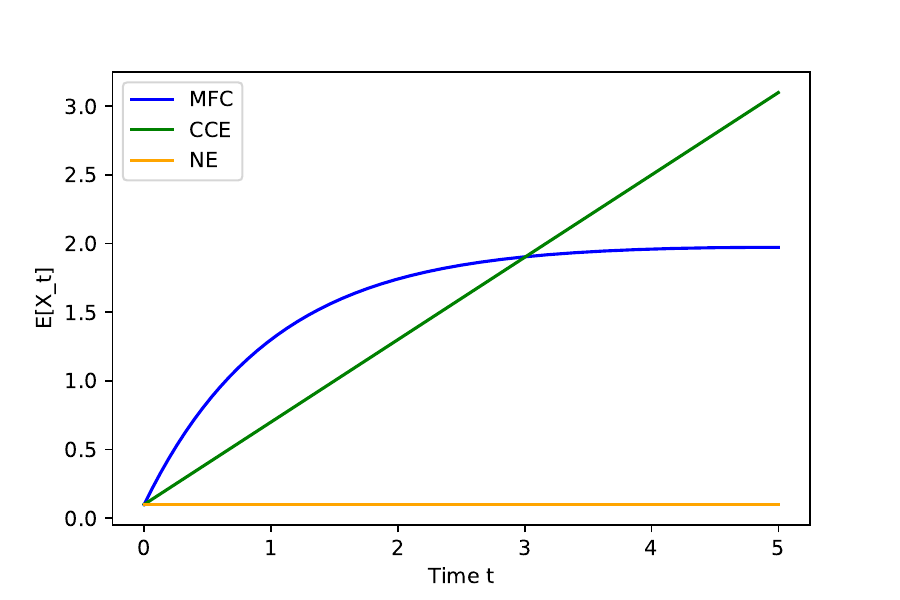}  
    \caption{Average level of cumulated abatement as a function of time for this same mean field CCE, the MFC solution and the mean field NE.}
  \label{fig:running_abat_effort_comp}
\end{subfigure}
\caption{A mean field CCE in $\mathcal G_l$ bridging the gap between the mean field NE and the MFC solution. Parameter values: $T = 5, a=2,\; b=1,\; \eps = 1,\; \nu_1 = 0.1$, $z_1 = 0.6,\;\sigma_z^2 = 0.06$. }
\label{fig:cce_bridging_gap}
\end{figure}

\begin{prop}\label{ag:signs_cM_cV}
The coefficients $c_M,c_V$ defined in Proposition \ref{ag:linear:opt_cond} verify the following:
\begin{enumerate}[label=(\roman*)]
\item $c_M < 0$,\label{item:sign_c_M}
\item $c_V> 0$ if and only if $\eps T^2 \geq 3$.
\end{enumerate}
\end{prop}
\begin{proof}
(i) We argue by contradiction.
Suppose $c_M \geq 0$.
By \eqref{ag:eq:linear:opt_cond}, this is equivalent to the existence of a mean field CCE in $\mathcal{G}_l$ so that the associated random variable $Z$ satisfies $\sigma^2_Z=0$ and $z_1>0$.
Since $\sigma^2_Z=0$, $(\lambda,\mu)$ is a mean field CCE with deterministic flow of moments $\mu_t = \nu_1 + tz_1$, for any $t$ in $[0,T]$.
By Theorem \ref{LQ_general:no_deterministic_cce}, this implies that $(\lambda,\mu)=(\hat{\alpha},\hat{m})$.
Since $\hat{m}_t = \nu_1$ for every time $t$, this implies that $z_1 = 0$, leading to a contradiction.

\medskip
(ii) We now show that $c_V> 0$ if and only if $\eps T^2 \geq 3$.
Since $c_M < 0$ by point \ref{item:sign_c_M}, condition \eqref{ag:eq:linear:opt_cond} implies that $c_V \geq 0$ if and only if there exists a mean field CCE in $\mathcal{G}_l$ so that the associated random variable $Z$ satisfies $z_1=0$ and $\sigma_z^2 >0$.
We now identify the conditions so that such a correlated flow is a mean field CCE.
In particular, it verifies $\E[\mu_t] =\nu_1$, for all $t\in [0,T]$.
By equation \eqref{ag:best_deviation}, the optimal strategy of the deviating player is given by 
$\hat\beta_t = \phi_t(\nu_1-\hat X_t)$, where $\hat X$ is deviating player's state process. Such a correlated flow is a mean field CCE if and only if $\J(\lambda,\mu)\geq \J(\hat\beta,\mu)$, which is in turn equivalent to $\J'(\lambda,\mu)\leq \J'(\hat\beta,\mu)$, where
\begin{align*}
    \J'(\hat\beta,\mu) &= \frac{1}{2}\E\left[\int_0^T\tonde{(\eps+\phi_t^2)\E[(\nu_1-\hat X_t)^2]+\eps t^2\sigma_z^2}dt\right],\\
\J'  (\lambda,\mu) &= \frac{1}{2}\E\left[\int_0^T\tonde{(\eps+\phi_t^2)\E[(\mu_t- X_t)^2] + \sigma_z^2}dt\right].
\end{align*}
By computing and comparing their dynamics, it can be shown that
\[
\E[(\mu_t- X_t)^2] = \E[(\nu_1-\hat X_t)^2], \quad \forall t \in [0,T].
\]
Therefore, $(\lambda,\mu)$ is a mean field CCE if and only if
$\eps\frac{T^3}{3}-T\geq 0$.
This allows to conclude that 
$c_V\geq 0$ is equivalent to $\eps T^2\geq 3$.
Since $c_V$ is null if and only if $\eps = 0$ and $\eps\neq 0$ by assumption, we deduce that $c_V> 0$ is equivalent to $\eps T^2 \geq 3$.
\end{proof}

Proposition \ref{ag:signs_cM_cV} implies that, if the reputational cost coefficient $\eps$ and time horizon $T$ are small enough, the only mean field CCE in $\mathcal G_l$ is the mean field NE. On the contrary, when $T,\eps$ are big enough,  for any expectation of $Z$ there exists a variance level so that any correlated flow with same expectation and higher variance is a mean field CCE.

\subsection{Comparison with mean field NE}
We have seen above that increasing the variance of $Z$ is a way to build mean field CCEs easily.
However, increasing the variance of $Z$ comes at the cost of lowering the odds to outperform the mean field NE, as shown in the next Proposition.
\begin{prop}[Outperformance over the mean field NE in $\mathcal G_l$]\label{ag:lin:outperf_condition}
    A correlated flow $(\lambda,\mu)\in \mathcal G_l$ outperforms the mean field NE in terms of payoff if and only if 
    \begin{equation}
        Tz_1(a- b\nu_1)-\tonde{z_1^2+\sigma_z^2}
        \left(b\frac{T^2}{3}+1\right) \geq 0
    \end{equation}
\end{prop}
\begin{proof}
This result follows directly from Proposition \ref{ag:outperf_condition}.
The inequality is assessed in the specific case of correlated flows in $\mathcal G_l$, using the following equalities:
\begin{equation}\label{ag:lin:values_moments_mu}
    \hat m_t = \nu_1,\quad \E[\mu_t] = \nu_1+tz_1,\quad \E[\mu_t^2] = \nu_1^2+2\nu_1tz_1+t^2(z_1^2+\sigma_z^2),\quad \frac{d\mu}{dt} = Z, \quad t\in [0,T].
\end{equation}
\end{proof}
The optimality and outperformance conditions for correlated flows in $\mathcal G_l$ in, respectively, Proposition \ref{ag:linear:opt_cond} and Proposition \ref{ag:lin:outperf_condition}, are both expressed in terms of the first and second moments of associated variable $Z$.
This allows us to characterize analytically a region of mean field CCEs outperforming the mean field NE in the plane $(z_1,\sigma_z^2)$.

\begin{prop}\label{ag:lin:outperf_region}
Assume $\eps T^2 \geq 3$.
Then, a correlated flow in $\mathcal G_l$ is a mean field CCE outperforming the mean field NE in terms of payoff if and only if the associated random variable $Z$ verifies
\begin{equation}\label{ag:lin:zone_outperf}
    -\frac{c_M}{c_V}z_1^2 \leq \sigma_z^2 \leq z_1\frac{3T(a-b\nu_1)}{bT^2+3}-z_1^2.
\end{equation}
Moreover, the set of mean field CCEs outperforming the mean field NE is not reduced to the mean field NE if and only if $a-b\nu_1 > 0$.
\end{prop}
\begin{proof}
By combining Propositions \ref{ag:cor:optimality_condition}, \ref{ag:signs_cM_cV} and \ref{ag:lin:outperf_condition}, we can see that a correlated flow in $\mathcal G_l$ with moments $z_1,\sigma_z^2$ for $Z$ is a mean field CCE outperforming the mean field NE in terms of expected payoff if and only if equation \eqref{ag:lin:zone_outperf} is verified.
Let us denote by $f,g:\R\to \R$ respectively the left hand-side and the right-hand side of that equation as function of $z_1$.
They are both parabola intersecting at the point $(0,0)$. The second derivative of $f$, $f''$, is strictly increasing as $c_V$ is positive according to Proposition \ref{ag:signs_cM_cV}, while $g''$ is strictly decreasing. Simple arguments therefore imply that the region between the two curves characterized in equation \eqref{ag:lin:zone_outperf}, i.e.,
\[ \{(x,y) \in \mathbb R^2: f(x)\leq y \leq g(x)\},\]
is not equal to the point $(0,0)$ if and only if $g'(0)> f'(0)$. This is the case if and only if $a-b\nu_1 > 0$. If the region between the two curves was reduced to the point $(0,0)$, the only mean field CCEs outperforming the mean field NE would verify $z_1=\sigma_z^2=0$, which corresponds to the mean field NE. 
\end{proof}

Figure \ref{fig:zone_outperf_CCE} represents the region of mean field CCEs outperforming the payoff of the mean field NE in the plane $(z_1,\sigma_z^2)$ for the same parameters as in Figure \ref{fig:cce_bridging_gap}. The outperformance condition parabola (``upper parabola'') is represented in red, while the optimality condition parabola (``lower parabola'') is in blue.

\begin{figure}[ht]
    \centering
    \includegraphics[width = 0.7\linewidth]{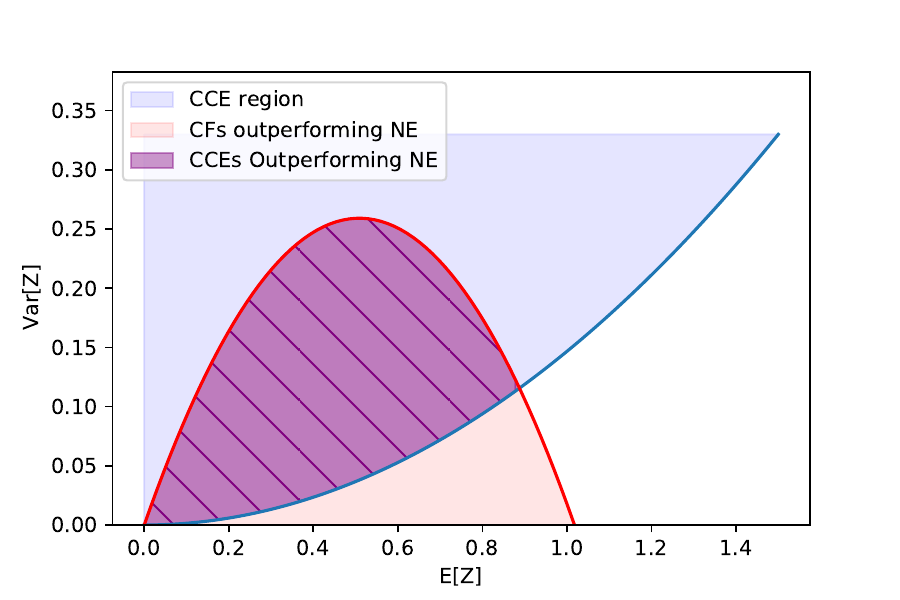}
    \caption{Region of mean field CCEs in $\mathcal G_l$ which outperform the Nash equilibrium in the plane $(z_1,\sigma_z^2)$. }
    \label{fig:zone_outperf_CCE}
\end{figure}

Proposition \ref{ag:lin:outperf_region} shows that each of the parameters $a$, $b$ and $\eps$ of the payoff plays specific roles in identifying mean field CCEs that outperform the payoff of the mean field NE.
The upper parabola comes from the outperformance condition and only depends on $a,b$ while the lower parabola comes from the optimality condition and only depends on $\eps$. The existence of the abatement benefit leaves space for more correlated flows to outperform the free-riding equilibrium payoff, as the upper parabola increases in $a$ and decreases in $b$. Moreover, Figure \ref{fig:minus_c_M_over_c_V} shows that the ratio $-c_M / c_V$ and hence the lower parabola is decreasing in $\eps$. Therefore, the reputational cost helps correlated flows to be CCEs. 
Indeed, with a higher reputational cost, countries have more interest in staying close to one another, and therefore the correlation device is more enforcing.

\begin{figure}[ht]
    \centering
    \includegraphics[width=0.5\linewidth]{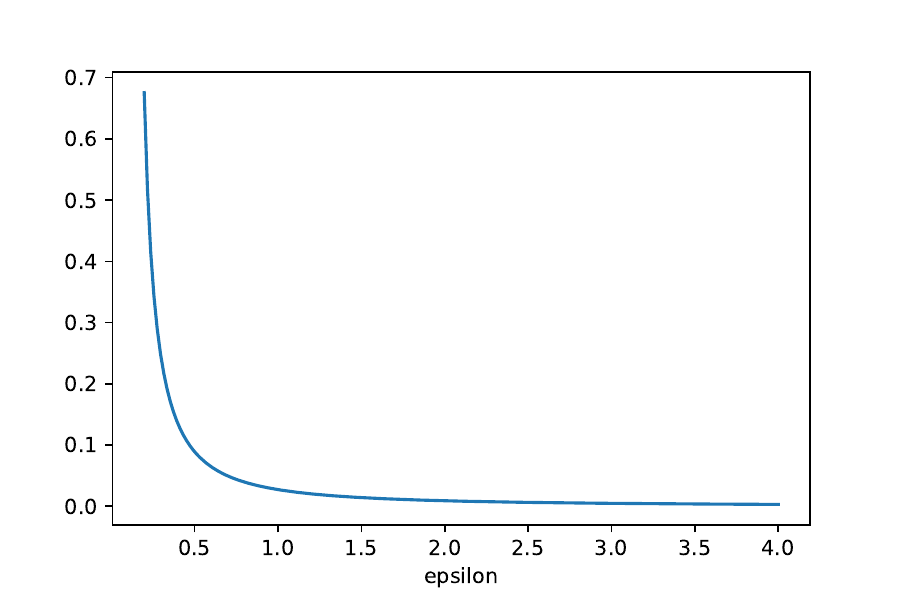}
    \caption{Representation of $-\sfrac{c_M}{c_V}$ in function of $\eps$, where $\eps$ verifies $\eps T^2 \geq 3$. Parameterization: $T=5$.}
    \label{fig:minus_c_M_over_c_V}
\end{figure}

Figure \ref{fig:3D_graph} represents the payoffs of the mean field CCEs belonging to $\mathcal G_l$ and which outperform the mean field NE in terms of payoff, i.e., verifying equation \eqref{ag:lin:zone_outperf}.
According to this graph, using the simple and tractable class of correlated flows $\mathcal G_l$, one is able to explore a large part of the payoffs attainable by mean field CCEs in this game.
Indeed, mean field CCEs payoffs get pretty close to the unattainable bound provided by the MFC solution payoff, relatively to the payoff of the mean field NE.

\begin{figure}[!ht]
    \centering
    \includegraphics[width = .8\textwidth]{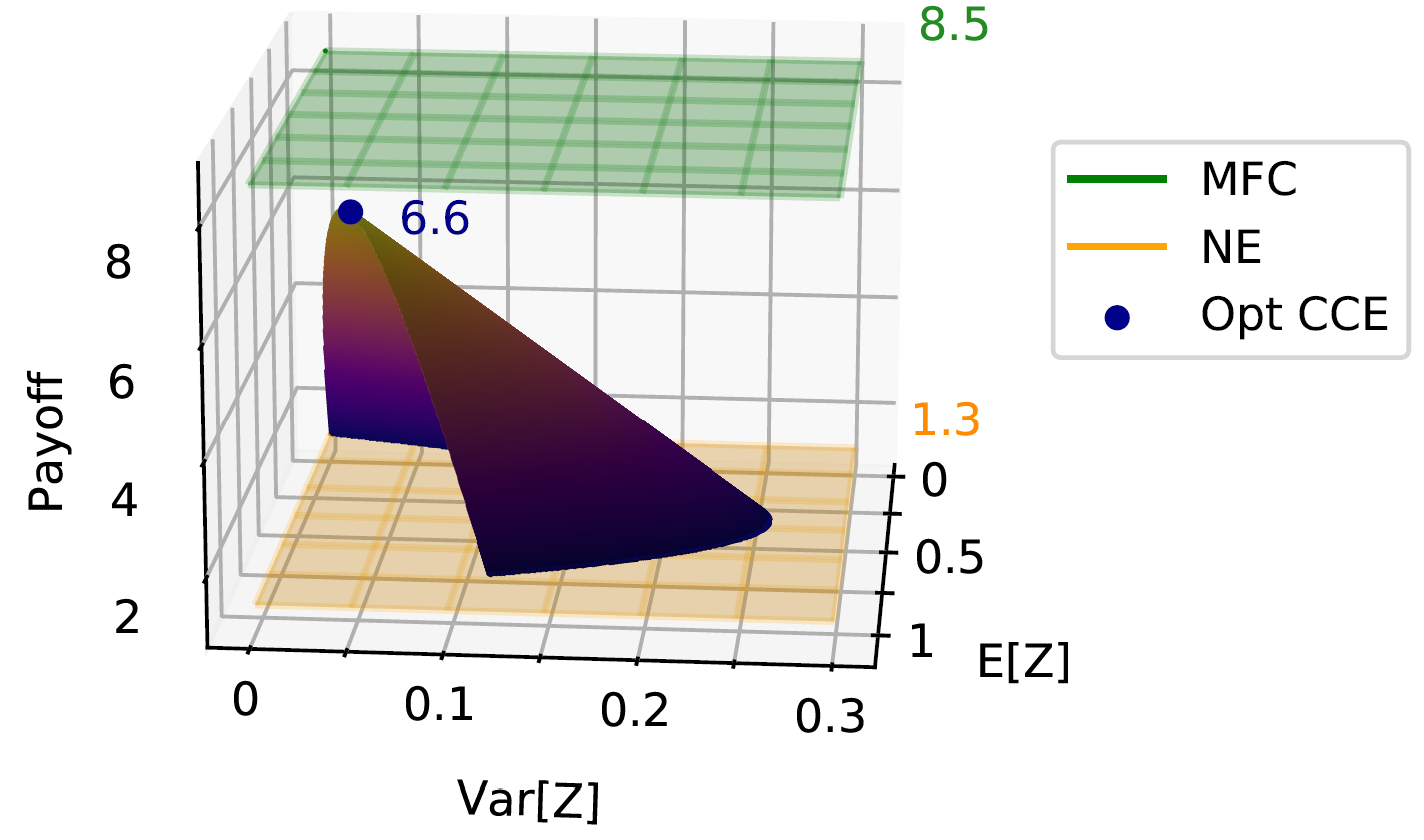}
    \caption{Expected payoff of mean field CCEs in $\mathcal G_l$ which outperform the mean field NE in terms of payoff (i.e. verifying equation \eqref{ag:lin:zone_outperf}) in the 3D space with $z_1$ as the x-axis, $\sigma_z^2$ as the y-axis and expected payoff as the z-axis. Parameters: same as Figure \ref{fig:cce_bridging_gap}.}
    \label{fig:3D_graph}
\end{figure}

The mean field CCEs in $\mathcal G_l$ which are optimal in terms of expected payoffs can be identified analytically.
\begin{prop}\label{ag:maximal_payoff}
Assume $ \eps T^2 \geq 3$.
Then, the expected payoff of mean field CCEs in $\mathcal G_l$  is maximised by a correlated flow $(\lambda,\mu)$ so that the associated random variable $Z$ satisfies
\begin{equation}\label{ag:maximal_payoff_condition}
    z_1 = \frac{T(a-b\nu_1)}{2(1-\frac{c_M}{c_V})\tonde{b\frac{T^2}{3}+1}}, \quad \sigma_z^2 = -z_1^2\frac{c_M}{c_V}.
\end{equation}
\end{prop}
\begin{proof}
We note that 
\begin{align*}
    \J(\lambda,\mu) &= \J(\hat{\alpha},m)  + \frac{T^2}{2}z_1(a- b\nu_1)-\frac{(\sigma_z^2+z_1^2)}{2}\left(b\frac{T^3}{3}+T\right).
\end{align*}
Therefore, $\J(\lambda,\mu)$ is strictly decreasing in $\sigma_z^2$.
Moreover, since $c_V>0$ according to Proposition \ref{ag:signs_cM_cV}, the optimality condition for correlated flows in $\mathcal{G}_l$ of Proposition \ref{ag:linear:opt_cond} implies that $(\lambda,\mu)$ is a mean field CCE if and only if 
\[
\sigma_z^2 \geq -z_1^2\frac{c_M}{c_V}.
\]
As $-\sfrac{c_M}{c_V}> 0$, for any given $z_1$ the mean field CCE with the highest expected payoff verifies $\sigma_z^2= -z_1^2\sfrac{c_M}{c_V}$. 
From now on, let us set $\sigma_z^2$ to this value.
We get
$$ \J(\lambda,\mu) = \J(\hat{\alpha},m)  + \frac{T^2}{2}z_1(a- b\nu_1)-\frac{z_1^2(1-\frac{c_M}{c_V})}{2}\left(b\frac{T^3}{3}+T\right),$$
which is a polynomial in $z_1$ whose maximum point is given by $z_1$ as in \eqref{ag:maximal_payoff_condition}.
\end{proof}

Figure \ref{fig:trade_off_obj} shows the payoffs of mean field CCEs in $\mathcal G_l$ with payoff-maximizing variance for $Z$, i.e. verifying $\sigma_z^2= -z_1^2\sfrac{c_M}{c_V}$. These payoffs are expressed as a function of $z_1$, on the $x$-axis, and they are compared to the payoffs of the MFC solution and the mean field NE, with same parameter settings as in the other figures. Figure \ref{fig:trade_off_ab} represents the average cumulated abatement over the whole time interval for the same equilibria, in the same fashion.

One can see out of Figure \ref{fig:trade_off_abat_obj} that for ``little ambitious'' mean field CCEs in $\mathcal G_l$, i.e., with relatively small $z_1$, there is actually a significant increase in both abatement levels and payoffs with regards to the mean field NE. However, there is a critical value of $z_1$, given by the payoff-maximising value of Proposition \ref{ag:maximal_payoff} and represented by the grey vertical line, after which increasing abatement comes at the cost of decreasing the payoff. This is in partial contrast with the results of \cite{DokkaTrivikram2022Edia} where a much stronger trade-off was observed. The difference is due to the presence of the reputational cost.

Characterizing a surface of mean field CCEs allows any moderator to choose the mean field CCE which corresponds to its goal, which might be to maximise expected payoff, or to consider positive externalities of abatement which are not ``priced'' in $\J$, and therefore to favor high abatement over maximising payoffs.

\begin{figure}[!ht]
\begin{subfigure}{.48\textwidth}
  \centering
\includegraphics[width =1.1\textwidth]{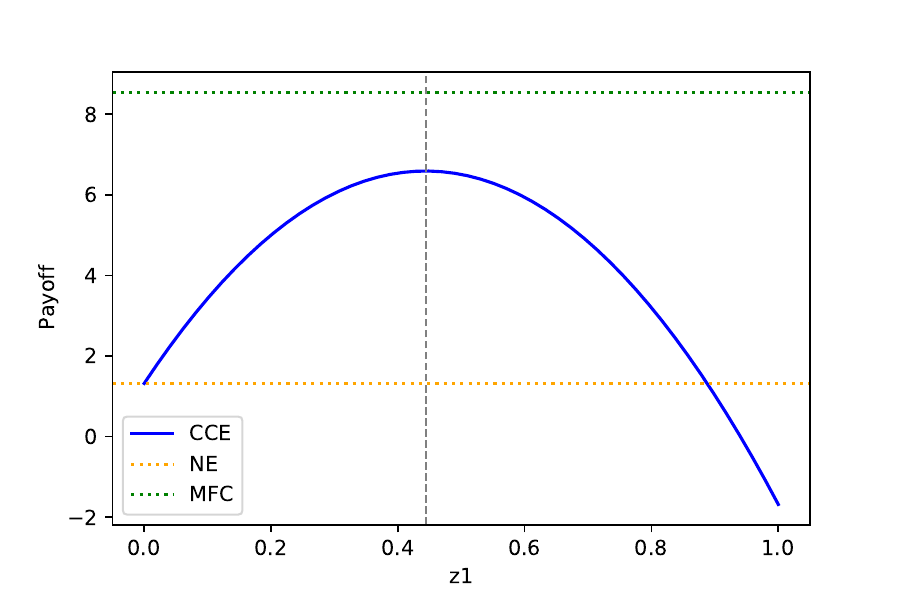}
    \caption{}
    \label{fig:trade_off_obj}
\end{subfigure}\hfill
\begin{subfigure}{.48\textwidth}
  \centering
    \includegraphics[width =1.1\textwidth]{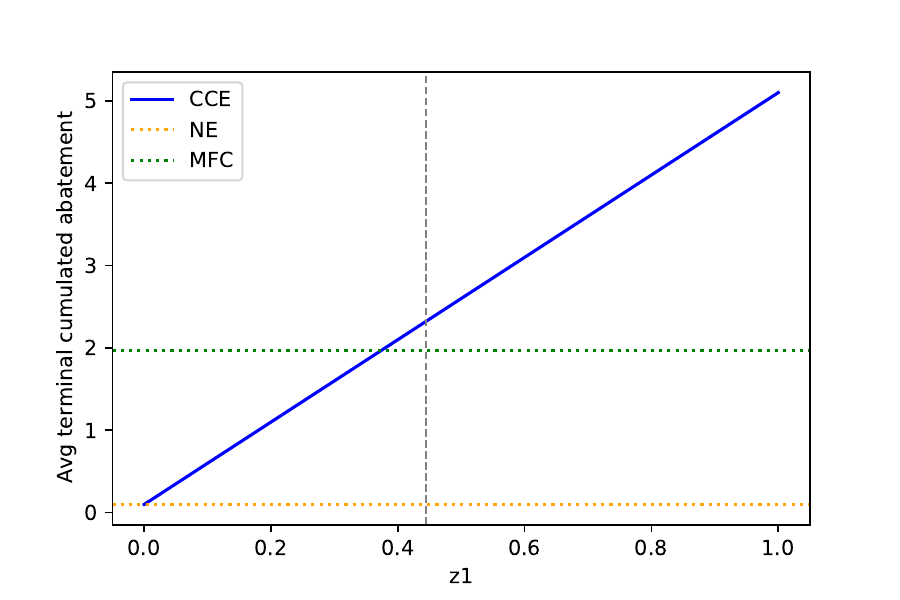}  
    \caption{}
  \label{fig:trade_off_ab}
\end{subfigure}
\caption{(a) Payoff of mean field CCEs of $\mathcal G_l$ with optimal variance as a function of $z_1$, and (b) their average cumulated abatement $\E[\mu_T]$, compared to the same quantities for the MFC solution and mean field NE. Parameter values are the same as in Figure \ref{fig:cce_bridging_gap}. }
\label{fig:trade_off_abat_obj}
\end{figure}

\addcontentsline{toc}{section}{Acknowledgements}
\section*{Acknowledgements}
The third author has been supported by French National Research Agency (ANR) under the program ``Investissements d’Avenir'' (``Investments for future''), via the FCD Labex of the Institut Louis Bachelier.

\addcontentsline{toc}{section}{Appendix}
\section*{Appendix}\label{sec:appendix}
\begin{proof}[Proof of Proposition \ref{LQ_general:prop_best_dev}]
We follow the approach of \cite[Chapter 6]{yong_zhou_stochastic_controls}.
We start by noticing that the equation for $\phi$ is a matrix Riccati equation, which admits a unique solution $\phi \in \mathcal{C}^1([0,T],\mathcal{S}^d)$ by Chapter 6, Theorem 7.2 therein. This implies the existence and uniqueness for $\psi$ and $\theta$ as well as they satisfy linear ODEs.
 
First, thanks to Assumptions \ref{LQ:assumptions}(4), the cost functional $\J'$ is strictly convex and therefore has a unique minimizer.
Indeed, by looking at \eqref{LQ:dev_player:only_cost_terms}, we have
\begin{align*}
\frac{1}{2} & \innprod{ Q_t X_t, X_t }  + \innprod{\Tilde{Q}_t \E[\mu_t] + q_t, X_t } +\frac{1}{2} \innprod{R_t\beta_t,\beta_t} + \innprod{S_t X_t, \beta_t } + \innprod{\Tilde{Q}_t \E[\mu_t] + q_t, X_t } + \innprod{ r_t,\beta_t} \\
& \geq \frac{1}{2}d_1 \abs{X_t}^2  + \frac{1}{2}d_2 \abs{\beta_t}^2 -\sup_t\abs{S_t}\abs{X_t}\abs{\beta_t} + \innprod{\Tilde{Q}_t \E[\mu_t] + q_t, X_t } + \innprod{ r_t,\beta_t} \\
& > \frac{1}{2}d_1 \abs{X_t}^2  + \frac{1}{2}d_2 \abs{\beta_t}^2 - \sqrt{d_1} \sqrt{d_2}\abs{S_t}\abs{X_t}\abs{\beta_t} + \innprod{\Tilde{Q}_t \E[\mu_t] + q_t, X_t } + \innprod{ r_t,\beta_t}.
\end{align*}
This inequality and the assumption $H \geq 0$ imply that the cost functional is strictly convex and lower semicontinuous, which yields that the minimizer exists and it is unique. Observe that this holds for any $(\E[\mu_t])_{t \in [0,T]}$, since it appears only in the linear terms $\innprod{\Tilde{Q}_t \E[\mu_t] + q_t, X_t }$ and $\innprod{ r_t,\beta_t}$.

\medskip
\noindent We apply the stochastic maximum principle, as in \cite[Chapter 6, Proposition 5.5]{yong_zhou_stochastic_controls}.
In the following, for the sake of clarity, we omit the dependence on time in all the matrices appearing in the coefficients and in the cost functions.
Let $\mathcal{H}(t,x,y,\beta)$ be the the reduced Hamiltonian of the system, defined as
\begin{align*}
    \mathcal{H} (t,x,y,\beta) & = \innprod{A x + B \beta,y} + \frac{1}{2}\innprod{Q x, x} + \innprod{\Tilde{Q}\E[\mu_t] + q,x}  + \innprod{Sx,\beta} + \frac{1}{2}\innprod{R\beta,\beta} + \innprod{r,\beta}.
\end{align*}
Then, the control $\hat{\beta}=(\hat{\beta}_t)_{t \in [0,T]}$ is optimal if and only if there exists a 4-tuple $(X,\hat{\beta},Y,Z)$ which satisfies
\begin{equation}\label{LQ:FBSDE}
\left\{ \begin{aligned}
    & dX_t = (A_t X_t + B_t \hat{\beta}_t)dt +\sigma_t dW_t, && X_0 = \xi, \\
    & dY_t = - \tonde{A_t^\top Y_t + Q_t X_t + \Tilde{Q}_t\E[\mu_t] + q_t + S_t^\top \hat{\beta}_t }dt + Z_t dW_t, && Y_T = H X_T + \Tilde{H}\E[\mu_T], \\
    & B_t^\top Y_t + S_t X_t + R_t \hat{\beta}_t + r_t = 0.
\end{aligned} \right.
\end{equation}
We make the following ansatz on $Y$:
\begin{equation*}
    Y_t = \phi_t X_t + \psi_t \E[\mu_t] + \theta_t,
\end{equation*}
with $\phi$, $\psi$ and $\theta$ deterministic functions taking values in $\R^{d \times d}$, $\R^{d \times d}$ and $\R^d$ respectively.
Since $R$ is invertible for every time $t$ by assumption \ref{LQ:assumptions}(4), by comparing the stochastic differential of the ansatz with \eqref{LQ:FBSDE}, we find that $Z_t = \hat{\phi}_t \sigma$ and that $\phi$, $\psi$ and $\theta$ must satisfy equations \eqref{LQ:riccati_eq}.
\end{proof}

\begin{proof}[Proof of Proposition \ref{MKV:prop_mkv_solution}]
We follow the Pontryagin maximum principle approach for MFC problems of \cite[Chapter 6]{librone_vol1}.
Let $\mathcal{H}$ be the Hamiltonian of the system:
\begin{equation*}
\begin{aligned}
    \mathcal{H} (t,x,y & ,m,\alpha) = \innprod{A x + B \alpha,y} - \frac{1}{2}\innprod{\Bar{Q} m, m} +\innprod{L,m} - \left( \frac{1}{2}\innprod{Q x, x} + \innprod{\Tilde{Q}m + q,x} \right. \\
    & \left. + \innprod{Sx,\alpha} + \frac{1}{2}\innprod{R\alpha,\alpha} + \innprod{r,\alpha} \right), \quad (t,x,y,m,\alpha) \in [0,T]\times\R^d\times\R^d\times\R^d\times\R^k.
\end{aligned}
\end{equation*}
Then, a control $\hat{\alpha}=(\hat{\alpha}_t)_{t \in [0,T]}$ is optimal if and only if there exists a 4-tuple $(X,\hat{\alpha},Y,Z)$ which satisfies
\begin{equation}\label{MFC:FBSDE}
\left\{ \begin{aligned}
    & dX_t  = (A_t X_t + B_t \hat{\alpha}_t)dt +\sigma dW_t, \\
    & X_0 = \xi, \\
    & dY_t  = - \tonde{A_t ^\top Y_t + Q_t X_t + q_t + S_t ^\top \hat{\alpha}_t + (\Bar{Q}_t +2\Tilde{Q}_t )\E[X_t] - L_t }dt  + Z_t dW_t, \\
    & Y_T = H X_t + (\Bar{H} + 2\Tilde{H})\E[X_T], \\
    & B_t ^\top Y_t + S_t X_t + R_t \hat{\alpha}_t  + r_t = 0.
\end{aligned} \right.
\end{equation}
Set $x_t = \E[X_t]$, $y_t = \E[Y_t]$ and $a_t = \E[\hat{\alpha}_t]$.
Then, by taking expectation, we get the following system
\begin{equation}\label{MFC:FBODE}
\left\{ \begin{aligned}
    & \dot{x}_t = A_t x_t + B_t a_t , && x_0 = \E[\xi], \\
    & \dot{y}_t = - \tonde{A_t ^\top y_t + (Q_t + 2\Tilde{Q}_t + \Bar{Q}_t) x_t + q_t -L_t + S_t ^\top a_t} , && y_T = (H + 2\Tilde{H} + \Bar{H} ) x_T , \\
    & B_t ^\top y_t + S_t x_t + R_t a_t + r_t = 0.
\end{aligned} \right.
\end{equation}
To find a solution, we make the following ansatz on $y$:
\begin{equation*}
    y_t = \phi^{MFC}_t x_t + \theta^{MFC}_t,
\end{equation*}
with $\phi^{MFC}$ and $\theta^{MFC}$ suitable deterministic functions taking values in $\R^{d \times d}$ and $\R^d$ respectively.
Since $R$ is invertible for every time $t$ by assumption \ref{LQ:assumptions}(4), by comparing the differential of the ansatz with \eqref{MFC:FBODE}, we get to equations \eqref{MFC:mean_riccati}.
By \cite[Chapter 6, Theorem 7.2]{yong_zhou_stochastic_controls} there exists a unique solution for the matrix Riccati equation for $\phi^{MFC} \in \mathcal{C}^1([0,T],\mathcal{S}^d)$.
We note that the flow of expectations $\bar{x}^{MFC}=(\bar{x}^{MFC}_t)_{t \in [0,T]}$ satisfies the differential equation \eqref{MFC:solution:flow_of_moments}.

\medskip
To prove the existence of a solution to the forward backward system \eqref{MFC:FBSDE}, we can make the ansatz
\begin{equation*}
    Y_t = \hat{\phi}_t X_t + \bar{\psi}_t \bar{x}^{MFC}_t + \bar{\theta}_t.
\end{equation*}
with $\hat{\phi}$, $\bar{\psi}$ and $\bar{\theta}$ deterministic functions taking values in $\R^{d \times d}$, $\R^{d \times d}$ and $\R^d$ respectively. By differentiating the ansatz, comparing it with \eqref{MFC:FBSDE} and using the invertibility of $R$ for any time $t$, we find that $Z_t = \hat{\phi}_t \sigma$ and that $\hat{\phi}$ satisfies the same equation as $\phi$, so that $\hat{\phi} = \phi$, and equations \eqref{MFC:riccati_control}  must be satisfied by $\bar{\psi}$ and $\bar{\theta}$.
\end{proof}

\begin{proof}[Proof of Proposition \ref{LQ:MFG:prop_mfg_solution}]
We follow the Pontryagin maximum principle approach together with the fixed point argument of \cite[Chapter 4]{librone_vol1}.
Let $\mathcal{H}$ be the Hamiltonian of the system:
\begin{equation*}
\begin{aligned}
    \mathcal{H} (t,x,y & ,m,\alpha) = \innprod{A x + B \alpha,y} - \frac{1}{2}\innprod{\Bar{Q} m, m} +\innprod{L,m} - \left( \frac{1}{2}\innprod{Q x, x} + \innprod{\Tilde{Q}m + q,x} \right. \\
    & \left. + \innprod{Sx,\alpha} + \frac{1}{2}\innprod{R\alpha,\alpha} + \innprod{r,\alpha} \right), \quad (t,x,y,m,\alpha) \in [0,T]\times\R^d\times\R^d\times\R^d\times\R^k.
\end{aligned}
\end{equation*}
Then, a control $\hat{\alpha}=(\hat{\alpha}_t)_{t \in [0,T]}$ is optimal if and only if there exists a 4-tuple $(X,\hat{\alpha},Y,Z)$ which satisfies
\begin{equation}\label{MFG:FBSDE}
\left\{ \begin{aligned}
    & dX_t = (A X_t + B \hat{\alpha}_t)dt +\sigma dW_t, && X_0 = \xi, \\
    & dY_t = - \tonde{A^\top Y_t + Q X_t + \Tilde{Q}\hat{m}_t + q + S^\top \hat{\alpha}_t }dt + Z_t dW_t, && Y_T = H X_t + \Tilde{H}\hat{m}_T, \\
    & B^\top Y_t + S X_t + R \hat{\alpha}_t  + r_t = 0.
\end{aligned} \right.
\end{equation}
Set $x_t = \E[X_t]$, $y_t = \E[Y_t]$ and $a_t = \E[\hat{\alpha}_t]$.
Then, the consistency condition $\E[X_t]=\hat{m}_t$ for every $0 \leq t \leq T$ holds if and only if the following system
\begin{equation}\label{MFG:FBODE}
\left\{ \begin{aligned}
    & \dot{x}_t = A x_t + B a_t , && x_0 = \E[\xi], \\
    & \dot{y}_t = - \tonde{A^\top y_t + (Q + \Tilde{Q}) x_t + q + S^\top a_t} , && y_T = (H + \Tilde{H} ) x_T , \\
    & B^\top y_t + S x_t + R a_t + r_t = 0,
\end{aligned} \right.
\end{equation}
admits a unique solution.
We make the following ansatz on $y$:
\begin{equation*}
    y_t = \phi^{NE}_t x_t + \theta^{NE}_t,
\end{equation*}
with $\phi^{NE}$ and $\theta^{NE}$ suitable deterministic functions taking values in $\R^{d \times d}$ and $\R^d$ respectively.
Since $R$ is invertible for every time $t$ by assumption \ref{LQ:assumptions}(4), by comparing the differential of the ansatz with \eqref{MFG:FBODE}, we get to equations \eqref{MFG:fixed_point_riccati}.
We note that the flow of moments $\hat{m}=(\hat{m}_t)_{t \in [0,T]}$ satisfies the differential equation \eqref{MFG:solution:flow_of_moments}.

\medskip
The last step it to prove the existence of a solution to the forward backward system \eqref{MFG:FBSDE}.
We make the ansatz
\begin{equation*}
    Y_t = \hat{\phi}_t X_t + \hat{\psi}_t \hat{m}_t + \theta^{\hat{m}}_t.
\end{equation*}
with $\hat{\phi}$, $\hat{\psi}$ and $\theta^{\hat{m}}$ deterministic functions taking values in $\R^{d \times d}$, $\R^{d \times d}$ and $\R^d$ respectively.
By the same reasoning of Proposition \ref{LQ_general:prop_best_dev}, we find that $Z_t = \hat{\phi}_t \sigma$, that $\hat{\phi}$ and $\hat{\psi}$ satisfy the same equations as $\phi$ and $\phi$, so that $\hat{\phi} = \phi$ and $\hat{\psi}=\psi$, and that equation \eqref{MFG:riccati_control} must be satisfied by $\theta^{\hat{m}}$.
\end{proof}

\begin{proof}[Proof of Theorem \ref{LQ_general:cce_mfg_comparison}]
By using $(\phi,\psi,\theta^{\hat{m}})$, we write the dynamics of the state process $X^{NE}$ as
\begin{align*}
    dX^{NE}_t & = \tonde{(A_t - B_tR_t^{-1}(B^\top \phi_t + S))X^{NE}_t -B_tR_t^{-1}( B^\top \psi_t \hat{m}_t +B^\top\theta^{\hat{m}}_t + r_t)}dt +\sigma_t dW_t \\
    & = \tonde{(A_t - B_t\Phi_t )X^{NE}_t - B_t( \Psi_t\hat{m}_t + \Theta_t ^{\hat{m}})}dt +\sigma_t dW_t,
\end{align*}
with $\Phi$ and $\Psi$ defined by \eqref{LQ:auxiliary_coefficients} and $\Theta^{\hat{m}}$ by \eqref{LQ_general:cce_mfg:auxiliary_theta}.
We remark that $\theta^{\hat{m}}$ and thus $\Theta^{\hat{m}}$ depend on $\hat{m}$ through its time derivative $\sfrac{d\hat{m}}{dt}$.
Let $f^{\hat{m},\mu}=(f^{\hat{m},\mu}_t)_{t \in [0,T]}$ be the solution of 
\begin{equation*}
\dot{f}^{\hat{m},\mu}_t = (A_t - B_t\Phi_t)f^{\hat{m},\mu}_t +B_t( \Psi_t(\hat{m}_t - \mu_t ) +\Theta^{\hat{m}}_t - \delta_t), \quad f^{\hat{m},\mu}_0 = 0.
\end{equation*}
By It\^o's formula, we have that
\begin{equation*}
    f^{\hat{m},\mu}_t = X_t - X^{NE}_t.
\end{equation*}
Since $f(\mu)$ is $\sigma(\mu)$-measurable and $X^{NE}$ is $\mathbb{F}^1$-progressively measurable, we have both that $X^{NE}$ and $f^{\hat{m},\mu}$ are independent and that
\begin{equation}\label{LQ_general:cce_mfg:equality_f}
    f^{\hat{m},\mu}_t =\E[f^{\hat{m},\mu}_t \vert \mu ] =\E[ X_t - X^{NE}_t \vert \mu ] = \mu_t-\hat{m}_t,
\end{equation}
by consistency condition.
Since it holds
\begin{align*}
\J & (\lambda,\mu) - \J(\hat{\alpha},\hat{m}) = \int_0^T \tonde{ \frac{1}{2}\innprod{\Bar{Q} \hat{m}_t, \hat{m}_t} - \frac{1}{2}\E[\innprod{ \Bar{Q} \mu_t, \mu_t  }] + \innprod{L_t, \E[\mu_t] - \hat{m}_t } }dt \\
& +\frac{1}{2}\innprod{\Bar{H}\hat{m}_T,\hat{m}_T} -\frac{1}{2}\E[\innprod{\Bar{H}\mu_T,\mu_T}] + \J' (\hat{\alpha},\hat{m}) -\J'(\lambda,\mu),
\end{align*}
we focus on the difference $\J'(\hat{\alpha},\hat{m})  -\J'(\lambda,\mu)$.
In a very similar way as in the proof of Theorem \ref{LQ:prop:optimality_condition} we obtain
\begin{align*}
\J'& (\hat{\alpha},\hat{m}) - \J' (\lambda,\mu) = \int_0^T \Big( \frac{1}{2} (\innprod{M_t \hat{m}_t,\hat{m}_t} - \E[\innprod{M_t (\hat{m}_t + f^{\hat{m},\mu}_t),\hat{m}_t + f^{\hat{m},\mu}_t}]) \\
& + \frac{1}{2}(\innprod{G_t \hat{m}_t,\hat{m}_t} - \E[\innprod{G_t \mu_t,\mu_t}]) + \innprod{N_t \hat{m}_t, \hat{m}_t - \E[\mu_t] } - \E[\innprod{N_t f^{\hat{m},\mu}_t, \mu_t }] \\
& - \innprod{q_t -\Phi_t^\top r_t, \E[f^{\hat{m},\mu}_t]}  -\innprod{\Psi_t^\top r_t,  \hat{m}_t - \E[\mu_t] } + \innprod{(R_t\Phi_t - S_t )\hat{m}_t, (\Theta^{\hat{m}}_t - R_t^{-1}\E[\delta_t] ) }  \\
& - \E[\innprod{(R_t\Phi_t - S_t )f^{\hat{m},\mu}_t, R_t ^{-1}\delta_t }] - \E[\innprod{R_t\Psi_t \mu_t, R_t ^{-1}\delta_t }]  \\
& + \innprod{R_t\Psi_t\hat{m}_t, \Theta^{\hat{m}}_t }  + \frac{1}{2}(\innprod{R_t\Theta^{\hat{m}}_t,\Theta^{\hat{m}}_t } - \E[\innprod{R^{-1}_t\delta_t ,\delta_t }]) - \innprod{r_t, \Theta^{\hat{m}}_t - R_t ^{-1}\E[\delta_t]} \Big) dt \\
&  +\frac{1}{2}\innprod{H\hat{m}_T, \hat{m}_T} - \frac{1}{2}\innprod{H (\hat{m}_T + \hat{f}_T(\mu)), \hat{m}_T + \hat{f}_T(\mu) } + \innprod{\Tilde{H}\hat{m}_T, \hat{m}_T - \E[\mu_T]} - \E[\innprod{\Tilde{H} \hat{f}_T(\mu), \mu_T }]. 
\end{align*}
Finally, we observe that, by using \eqref{LQ_general:cce_mfg:equality_f}, we have 
\begin{align*}
& \innprod{M_t \hat{m}_t,\hat{m}_t} - \E[\innprod{M_t (\hat{m}_t + f^{\hat{m},\mu}_t),\hat{m}_t + f^{\hat{m},\mu}_t}] = \innprod{M_t \hat{m}_t,\hat{m}_t} - \E[\innprod{M_t \mu_t,\mu_t}], \\
& \innprod{N_t \hat{m}_t, \hat{m}_t - \E[\mu_t] } - \E[\innprod{N_t f^{\hat{m},\mu}_t, \mu_t }] = \innprod{N_t \hat{m}_t, \hat{m}_t } - \E[\innprod{N_t \mu_t, \mu_t }], \\
& \innprod{(R_t\Phi_t - S_t)\hat{m}_t, (\Theta^{\hat{m}}_t - R_t^{-1}\E[\delta_t] ) }  - \E[\innprod{(R_t\Phi_t - S_t)f^{\hat{m},\mu}_t, R_t ^{-1}\delta_t }] \\
& =  \innprod{(R_t\Phi_t - S_t)\hat{m}_t, \Theta^{\hat{m}}_t  }  - \E[\innprod{(R_t\Phi_t - S_t)\mu_t, R_t ^{-1}\delta_t }], \\
& \innprod{q_t -\Phi_t^\top r_t, \E[f^{\hat{m},\mu}_t]}  +\innprod{\Psi_t^\top r_t,  \hat{m}_t - \E[\mu_t] } = \innprod{q_t -(\Phi_t+\Psi_t)^\top r_t,  \E[\mu_t] -\hat{m}_t }.
\end{align*}
By using these identities together with \eqref{LQ_general:optimality:final_identities}, we get to \eqref{LQ:cce_mfg_outperformance:final}.
\end{proof}

\addcontentsline{toc}{section}{References}

\bibliographystyle{abbrv}
\bibliography{biblio}

\end{document}